\documentclass[10pt]{article}
\usepackage{amsmath,fullpage,amsthm}
\usepackage{amscd}
\usepackage{amsthm,mathtools} \usepackage{amssymb}
\usepackage{latexsym}
\usepackage{eufrak}
\usepackage{euscript}
\usepackage{epsfig}
\usepackage{tikz}
\usepackage{graphics}
\usepackage{array}
\usepackage{mathabx}

\usepackage{todonotes}

\usepackage{ragged2e}

\usepackage{enumerate} \usepackage{authblk}
\theoremstyle{plain}
\newtheorem{theorem}{Theorem}[section]
\newtheorem{corollary}{Corollary}[section]
\newtheorem{lemma}{Lemma}[section]
\newtheorem{proposition}{Proposition}[section]
\newtheorem{problem}{Problem}[section]

\newtheorem{remark}{Remark}[section]

\newcommand{\PSL}{\mathop{\mathrm{PSL}}}
\newcommand{\PGL}{\mathop{\mathrm{PGL}}}
\newcommand{\PSU}{\mathop{\mathrm{PSU}}}

\newcommand{\PSp}{\mathop{\mathrm{PSp}}}

\begin{document}
\justifying
%%%%%%%%%%%%%%%%%%%%%%%%%%%%%%%%%%%%%%%%%%%%%%%%%%%%%%%%%%%%%%%%%%%%%%%%%%%%%%%%%%%%%%%%%%%%%%%%%%%%%%%%%%%%%%%%%%%%

  \setcounter{Maxaffil}{3}
   \title{Intersection subgroup graph with forbidden subgraphs}
      \author[a ]{Santanu Mandal\thanks{santanu.vumath@gmail.com}}

\author[ b]{Pallabi Manna\thanks{mannapallabimath001@gmail.com}}
 \affil[ a]{Department of Mathematics,}
   \affil[ ]{National Institute of Technology,}
   \affil[ ]{Rourkela - 769008, India}
 %\affil[ a]{Department of Mathematics,}
   \affil[b ]{Harish Chandra Research Institute,}
   \affil[ ]{Prayagraj - 211019, India}
  
   \maketitle

\begin{abstract}
Let $G$ be a group. The intersection subgroup graph of $G$ (introduced by Anderson et al. \cite{anderson}) is the simple graph $\Gamma_{S}(G)$ whose vertices are those non-trivial subgroups say $H$ of $G$ with $H\cap K=\{e\}$ for some non-trivial subgroup $K$ of $G$; two distinct vertices $H$ and $K$ are adjacent if and only if $H\cap K=\{e\}$, where $e$ is the identity element of $G$. In this communication, we explore the groups whose intersection subgroup graph belongs to several significant graph classes including cluster graphs, perfect graphs, cographs, chordal graphs, bipartite graphs, triangle-free and claw-fee graphs. We categorise each nilpotent group $G$ so that $\Gamma_S(G)$ belongs to the above classes. We entirely classify the simple group of Lie type whose intersection subgroup graph is a cograph. Moreover, we deduce that $\Gamma_{S}(G)$ is neither a cograph nor a chordal graph if $G$ is a torsion-free nilpotent group.
\end{abstract}

%NOTE NOTE NOTE NOTE%

%\todo[inline]{Sir, we are interested to compute the perfectness of $\Gamma_{S}(G)$ when $G$ is a $p$-group or $SL(2, q)$ or $SL(3, q)$ or some low-dimensional simple lie groups, sporadic simple groups etc.(if possible). But currently we are unable to solve this; still we keep  trying to answer this portion. We will be very gratefull if you help us in this direction. I attached a paper written by Britnell \& Gill on title "Perfect Commuting graph". We are trying to get such type of characterizations of groups. Also, we are facing difficulties to compute whether $\Gamma_S(SL(2,q)),~\Gamma_S(SL(3,q))$ is cograph or not, chordal graph or not.} 

%NOTE NOTE NOTE NOTE%

\section{Introduction}
In algebraic graph theory, graphs defined on groups have been paid significant attention over the past few decades. Group-derived graphs display excellent network properties, such as large girth, which are crucial for determining how connected the network is. These graphs also have a wide range of uses, such as in automata theory (see \cite{K_1, K_2}). The intersection subgroup graph is yet another fresh addition to this collection. The intersection subgroup graph of a group $G$ is defined as the vertex set consists all such non-trivial subgroups (say $H$) of $G$ such that $H\cap K=\{e\}$ (where, $K$ is any non-trivial subgroup of $G$) and two distinct vertices $M, N$ are connected by an edge if and only if $M\cap N=\{e\}$, where $e$ is the identity element of the group $G$. This graph contributes very well to comprehend a group's subgroup structures. Anderson et al. \cite{anderson} introduced the idea of construction of such a nice graph. They studied various graph theoretic properties with an emphasis on connectedness, completeness, planarity, complete bipartiteness etc. of the intersection subgroup graph. In the same paper, the authors proved that the intersection subgroup graph is subgroup-closed i.e., if $H$ is a subgroup of $G$ then $\Gamma_{S}(H)$ is an induced subgraph of $\Gamma_{S}(G)$. This subgroup-closed property of the intersection subgroup graph helps us a lot in our study. \\
In graph theory, there are several graph classes including planar graphs, cluster graphs, perfect graphs, cographs, chordal graphs, line graphs etc that can be represented in terms of forbidden subgraphs. In our study, we cover the classes of cluster graphs, perfect graphs, cographs, chordal graphs, bipartite graphs, triangle-free and claw-fee graphs. The definitions of each of these graphs are given in the respective section.\\
In this paper, we fully characterize the nilpotent groups whose intersection subgroup graph is a perfect graph, a triangle-free graph, cluster graph, claw-free graph, cograph, chordal graph.\\
Analogously, we completely determine the simple group $G$, a group of Lie type of low rank, such that $\Gamma_{S}(G)$ is a cograph. In this case, we conclude that $G$ is one of the group: $ \PSL(2, 2), \PSL(2, 3)$ and 
$G={}^{2}{B_{2}}(q)$ where $q=2^{2e+1}(q \geq 8)$ such that the numbers $q-1,~ q\pm \sqrt{2q}+1$ are either prime power or of the form $p^{a}q^{b}(a, b\geq 1)$.\\
Moreover, we prove that there is no sporadic simple group $G$ for which $\Gamma_{S}(G)$ is either a cograph or a chordal graph. We also analyze the symmetric group ($S_{n}$), the alternating group ($A_{n}$), the dihedral group ($D_{n}$) such that their intersection graph is a cograph, a chordal graph. \\
In the last section, we consider the torsion-free nilpotent group. A group is called torsion-free if its every non-identity elements are of infinite order. We conclude that the intersection subgroup graph of a torsion-free nilpotent group is neither a cograph nor a chordal graph.\\
In this paper, the notation $P_{n}$ stands for a path on $n$-vertices. We use the notation $C_{n}$ for a cyclic group of order $n$ as well as a cycle of length $n$. The context will be clear in each case which is intended. Now we want to recall the definition of an induced subgraph of a graph. The induced subgraph $\Gamma[S]$ of a graph $\Gamma=(V, E)$ is the graph whose vertex set is $S\subset V$ and edge set consists all those edges of $E$ that have both end points in $S$. If a graph does not contain $H$ as its induced subgraph then the graph is called a $H$-free graph.\\
Before moving on to the next section, we want to recall a theorem from group theory which we used.
\begin{theorem}\textit{(Kulakoff Theorem)}\\
Let $G$ be a group of order $p^{\alpha}$. Then\\
a) the number of subgroups of prime power order is congruent to $1$ (mod $p$),\\
b) if $G$ has unique subgroup of order $p^{\beta}$ for all $\beta$ satisfying $1<\beta \leq \alpha$, then $G$ is cyclic or $\beta =1$ and $p=2$, $G$ is the generalized quaternion group $Q_{2^{\alpha}}$.
\end{theorem}
\section{Perfect Graph}
A finite graph $\Gamma$ is called perfect if it has the property that every
induced subgraph has clique number equal to chromatic number.\\
The class of perfect graphs is closed under complementation (this is the
\emph{Weak Perfect Graph Theorem} of Lov\'asz, and contains
several further important graph classes, including bipartite graphs,
line graphs, chordal graphs, cographs etc. First we recall a theorem: 
\begin{theorem}\emph{(Strong Perfect Graph Theorem)}
\label{spgt}
A graph is perfect if and only if
it contains no odd hole or odd antihole as induced subgraph, where an
odd hole is an $n$-cycle with $n$ odd and $n\ge5$, and an
odd antihole is the complement of an odd hole.
\end{theorem}
The above theorem is called \emph{Strong Perfect Graph Theorem}, a conjecture of Berge proved by Chudnovsky, Robertson, Seymour and
Thomas \cite{spgt}.\\
We use this theorem to identify the all nilpotent groups $G$ such that $\Gamma_{S}(G)$ is a perfect graph.
\begin{lemma}
\label{lm_cy}
Let $G\cong C_{p^a}\times Q$ (where, $Q$ is any finite non-cyclic $q$-group). Then, $\Gamma_{S}(G)$ is perfect.
\end{lemma}

 \begin{proof}
First, we show that $\Gamma_{S}(G)$ does not contain any $5$-cycle. Suppose, $(A, B, C, D, E, A)$ is a $5$-cycle. Without loss of generality, let $o(A)$ is power of $p$. Then, $o(B), o(D)$ must be power of $q$ such that $B\cap D$ is non-trivial. Since, $C\nsim E$ and $C\nsim A$, therefore $o(C)$ must be divisible by $pq$. Similarly, $pq$ divides $o(D)$. As the Sylow $p$-subgroup of $G$ is cyclic, therefore $C$ and $D$ has common intersection, which is at least $C_{p}$. This implies, $C$ and $D$ can not be adjacent. Hence such $5$-cycle does not exist in $\Gamma_{S}(G)$. Next, we consider the higher cycle say $A_{1}, A_{2}, \cdots, A_{2n}, A_{2n+1}, A_{1}$. Without loss of generality, let $o(A_{1})$ is power of $p$. So, $o(A_{2})$ and $o(A_{2n+1})$ are power of $q$. Then $p$ divides $o(A_{3})$ and $o(A_{2n})$ as well. Now, as $A_{1}\nsim A_{4}, A_{2n-1}$ so $o(A_{4}), o(A_{2n-1})$ must be divisible by $p$, which contradicts that $A_{3}\sim A_{4}, A_{2n}\sim A_{2n-1}$. So, such cycle does not exist. Similarly, for the other possibilities of $o(A_{1})$ we also get that $(n+1)-$th and $(n+2)-$th vertices are not adjacent. This happens because the Sylow $p$-subgroup is cyclic and $Q$ is finitely generated.\\
Therefore, $\Gamma_{S}(G)$ does not contain an odd cycle of length $5$ and more. Using the similar procedure we can show that the complement of $\Gamma_{S}(G)$ can never contains any odd cycle of length $5$ and above. Thus, $\Gamma_{S}(G)$ is perfect [by Theorem \ref{spgt}].
\end{proof}

\begin{lemma}
\label{lm_gen_1}
Let $G$ be any one of the groups: $Q_{2^a}\times C_{q^b}$ or $Q_{2^a}\times C_{q^{b}r^{c}}$ or $Q_{2^a}\times C_{q^{b}r^{c}s^{d}}$, where $a\geq 3$, $b, c, d \geq 1$ and $q, r$ are odd primes. Then $\Gamma_{S}(G)$ is not perfect.
\end{lemma}
\begin{proof}
Let $G\cong Q_{2^a}\times C_{q^b}$. Let $A_{1}, A_{2}, \cdots , A_{2n+1}, A_{1}$ ($n\geq 2$) be a cycle in $\Gamma_{S}(G)$. Clearly, any two subgroups of order power of $2$ must be non-adjacent as they share a non-trivial intersection. Without loss of generality, let $o(A_{1})$ is power of $q$. Then, $o(A_{2}), o(A_{2n+1})$ are power of $2$. So, $q$ divides $o(A_{3})$ and $o(A_{2n})$. But as $A_{4}, A_{2n-1}\nsim A_{1}, A_{2}$ so $o(A_{4}), o(A_{2n-1})$ are divisible by $2q$. This leads to a contradiction that $A_{3}\sim A_{4}$ and $A_{2n}\sim A_{2n-1}$. So, such odd cycle does not exist in $\Gamma_{S}(G)$. Similarly, it can be shown that the complement of $\Gamma_{S}(G)$ does not contain $C_{2n+1}$( for $n\geq 2$). Hence, $\Gamma_{S}(G)$ is perfect.\\
Using a similar strategy, we can demonstrate that $\Gamma_{S}(G)$ is a perfect graph if $G$ is any one of the groups $Q_{2^a}\times C_{q^{b}r^{c}}$ or $Q_{2^a}\times C_{q^{b}r^{c}s^{d}}$, where $a\geq 3$ and $b, c, d \geq 1$. 
\end{proof}

\begin{lemma}
\label{lm_gen_2}
$\Gamma_{S}(Q_{2^a}\times K)$ is perfect, where $K$ is any finite non-cyclic $q$-group for an odd prime $q$.
\end{lemma}
\begin{proof}
Note that, as $K$ is a finite $q$-group so it is finitely generated. The proof is as similar as Lemma \ref{lm_gen_1}.
\end{proof}

\begin{lemma}
\label{lm_gen_3}
$\Gamma_{S}(Q_{2^a}\times C_{q^b}\times R)$ is perfect, where $R$ is a non-cyclic (finite) group of order power of $r$ and $q, r$ are odd primes.
\end{lemma}
\begin{proof}
The strategy of the proof is similar as Lemma \ref{lm_cy} and Lemma \ref{lm_gen_1}.
\end{proof}

\begin{lemma}
\label{lm_pq}
Let $G$ be the group $P\times Q$, where $a, b \geq 1$ and $P, Q$ be two finite non-cyclic $p$-group and $q$-group respectively for odd primes $p, q$. Then $\Gamma_{S}(G)$ is perfect.
\end{lemma}

\begin{proof}
Since, $p, q$ both odd and $P, Q$ are finite non-cyclic $p$ and $q$ groups so both $P, Q$ are at least $2$-generated. Let both $P, Q$ be $2$-generated and $P_{1}, P_{2}\in P$, $Q_{1}, Q_{2}\in Q$ be the distinct $p$-subgroups and $q$-subgroups of $G$. Suppose, $\Gamma_{S}(G)$ contains an odd cycle $A_{1}, A_{2},\cdots, A_{2n+1}, A_{1}$, where $n \geq 2$. Without loss of generality, let $A_{1}\cong P_{2}\times Q_{1}$. Then, $A_{2}\cong P_{2}\times Q_{2}, A_{2n+1}\cong P_{2}\times \{e\}$. So, $A_{2n-1}\cong P_{1}\times Q_{2}$ and $A_{4}\cong P_{2}\times Q_{1}$; but this contradicts $A_{3}, A_{4}$ are adjacent. Similarly for the other possiblities we will reach a contradiction. Therefore, $\Gamma_{S}(G)$ does not contain any cycle of length $2n+1$( for $n \geq 2$).\\
Similarly, one can prove that the complement of $\Gamma_{S}(G)$ does not contain any cycle of length $5$ and above.\\
This proves that $\Gamma_{S}(G)$ is perfect according to the Theorem \ref{spgt}.
\end{proof}

\begin{theorem}
\label{th_nil_perfect}
Let $G$ be a finite nilpotent group which is not a $p$-group. Then $\Gamma_{S}(G)$ is a perfect graph if and only if the following conditions hold:\\
a) $|\pi(G)|=4$ and $G\cong C_{p^{a}q^{b}r^{c}s^{d}}$, where $a, b, c, d \geq 1$, or $Q_{2^a}\times C_{q^{b}r^{c}s^{d}}$( where, $a \geq 3$ and $b, c, d \geq 1$);\\ 
b) $|\pi(G)|=3$ and $G$ is either a cyclic group or $Q_{2^a}\times C_{q^{b}r^{c}}$ ($a\geq 3$ and $b,c \geq 1$) or $C_{2^{a}q^{b}}\times R$ ($a, b\geq 1$) or $Q_{2^a}\times C_{q^b}\times R$ ($a \geq 3$), where $R$ is a non-cyclic (finite)group of order power of $r$ ;\\
c) $|\pi(G)|=2$ and $G$ is any one of the group  $C_{p^{a}q^{b}}$ or $C_{p^a}\times Q$ or $P\times C_{q^b}$ or $P\times Q$, where $a, b \geq 1$ and $P, Q$ are finite non-cyclic $p$-group and $q$-group respectively ($p, q$ are odd primes), or $Q_{2^a}\times C_{q^b}$ or $Q_{2^a}\times K$ such that $a\geq 3, b\geq 1$ and $K$ is a non-cyclic (finite) group of order power of $q$.
\end{theorem}

\begin{proof}
Firstly, let $G$ be a finite nilpotent group such that $o(G)$ has $5$ distinct prime divisors say $p_{1}, p_{2}, p_{3}, p_{4}$ and $p_{5}$.  Suppose, $P_{i}$ be the Sylow $p_{i}$-subgroups of $G$ for $1\leq i\leq 5$. Then $\Gamma_{S}(G)$ comprises a $5$-cycle $(P_{1}P_{2}, P_{3}P_{4}, P_{2}P_{5}, P_{1}P_{4}, P_{3}P_{5}, P_{1}P_{2})$. Thus $o(G)$ can have at most $4$ distinct prime divisors. \\
\textbf{Case 1.} Let $p, q, r, s|o(G)$, where $p<q<r<s$ are primes.\\
Let $P, Q, R, S$ be the Sylow $p$, $q$, $r$, $s$-subgroups of $G$ respectively. Here we consider two cases based on $p$.\\
\textbf{Subcase 1.} Let $p$ be odd\\
 \textbf{We claim that no Sylow subgroup of $G$ is non-cyclic}. \\
If possible, let $P$ is non-cyclic. Then there exist at least two $p$-subgroups say $P_{1}, P_{2}$ such that $P_{1}\cap P_{2}=\{e\}$. But, $\Gamma_{S}(G)$ contains a $5$-cycle: $(QR, P_{2}S, P_{1}R, P_{2}Q, P_{1}S, QR)$. Therefore, $G \cong C_{p^{a}q^{b}r^{c}s^{d}}$, where $a, b, c, d \geq 1$.\\
\textbf{Subcase 2.} $p=2$\\
Clearly, none of the Sylow $q$, $r$, $s$-subgroups be cyclic; elsewhere we get a $5$-cycle which contradicts $\Gamma_{S}(G)$ is perfect. So, $G\cong P\times C_{q^{b}r^{c}s^{d}}$.\\
\textbf{We assert that $P$ must be either cyclic or generalized quaternion group $Q_{2^a}$, where $a\geq 3$}.\\
Let, $P$ is neither cyclic nor a generalized quaterion group $Q_{2^a}$. Then $P$ contains at least two subgroups of order $2$ and by the same argument in Subcase-1 we get a $5$-cycle.\\
Thus $G$ is either $C_{2^{a}q^{b}r^{c}s^{d}}$(where, $a, b, c, d \geq 1$ or $Q_{2^a}\times C_{q^{b}r^{c}s^{d}}$( where, $a \geq 3, b, c, d \geq 1$).\\
\textbf{Case 2.} Let $o(G)$ has $3$ distinct prime divisors.\\
Suppose, $p< q<r$ are $3$ distinct prime divisors of $o(G)$ and $P, Q, R$ are the corresponding Sylow subgroups of $G$ repectively. Now, we consider two cases such that either $p=2$ or $p \neq 2$.\\
\textbf{Subcase 1.} $p\neq 2$\\
\textbf{We assert that only one Sylow subgroup of $G$ can be non-cyclic}.\\
For the sake of contradiction, we assume that $G$ has two non-cyclic Sylow subgroups say $P, Q$. Then $\Gamma_{S}(G)$ carries a $5$-cycle $(P_{1}Q_{1}, Q_{2}R, P_{2}Q_{1}, P_{1}Q_{2}, P_{2}R, P_{1}Q_{1})$, where $P_{1}, P_{2}$ and $Q_{1}, Q_{2}$ are two distinct $p$ and $q$-subgroups of $G$ respectively that intersects trivially. Thus, $G$ is either a cyclic group or a non-cyclic group such that only one Sylow subgroup of $G$ is non-cyclic.\\
\textbf{Subcase 2.} $p=2$\\
Clearly, both the Sylow $q$ and $r$ subgroup cannot be non-cyclic.Otherwise, we get a $5$-cycle. \\
Just as done in Case 1 (Subcase 2), we can conclude that $G$ is either a cyclic group $C_{2^{a}q^{b}r^{c}}$ or $Q_{2^a}\times C_{q^{b}r^{c}}$ ($a\geq 3, b,c \geq 1$) or $C_{2^{a}q^{b}}\times R$ ( where, $a, b\geq 1$) or $Q_{2^a}\times C_{q^b}\times R$ ( $a \geq 3$), where $R$ is a non-cyclic group of order power of $r$.\\
\textbf{Case 3.} $o(G)$ has only two distinct prime prime divisors.\\
Let $p<q$ be two distinct primes divides $o(G)$. If $p$ is odd, then $G$ is either $C_{p^{a}q^{b}}$ or $C_{p^a}\times Q$ or $P\times C_{q^b}$ or $P\times Q$, where $a, b \geq 1$ and $P, Q$ are non-cyclic $p$-group and $q$-group respectively.\\
If $p=2$, then we get the possiblities of $G$ are either $Q_{2^a}\times C_{q^b}$ or $Q_{2^a}\times K$ such that $a\geq 3, b\geq 1$ and $K$ is a non-cyclic group of order power of $q$.\\
Converse Part:
Let $G \cong C_{p^{a}q^{b}r^{c}s^{d}}$. First we show that, $\Gamma_{S}(G)$ does not contain a $5$-cycle. If possible let, there exist a $5$-cycle with vertices $A, B, C, D, E$. Without loss of generality, let $pq|o(A)$. Then $o(B)$ is power of $r$ and $o(E)$ is $r^{l}s^{m}$. So, $o(C)$ must be $p^{k}s^{n}$. But $o(D)$ must be either divisible by $r$ or $B \sim D$. If $r|o(D)$ this contrdaicts $D\sim E$. Thus, in any aspect one can show that $\Gamma_{S}(G)$ does not contain any cycle of length $5$. Similarly we can prove that neither $\Gamma_{S}(G)$ nor its complement does not have any odd cycle of length $5$ and above, which proves that $\Gamma_{S}(G)$ is perfect.\\
Following a similar procedure, we can prove that if $G$ is any one of the cyclic group in b) or c) then neither $\Gamma_{S}(G)$ nor the complement of $\Gamma_{S}(G)$ contains any odd cycle of length $5$ and above. This proves the perfectness of $\Gamma_{S}(G)$.\\
For the other groups, the proof follows from Lemmas \ref{lm_cy}, \ref{lm_gen_1}, \ref{lm_gen_2}, \ref{lm_gen_3}, \ref{lm_pq}.
\end{proof}

\section{Triangle-free}
A graph $\Gamma$ is called triangle-free if $\Gamma$ does not contain any triangle as its induced subgraph. \\
The triagle-free graphs are the subclass of the bipartite graphs. In this section we investigate the finite nilpotent groups whose intersection subgroup graph is triangle-free as well as bipartite graph.
\begin{theorem}
\label{th_nil_triangle}
Let $G$ be a finite nilpotent group. Then $\Gamma_{S}(G)$ is triangle-free if and only if $G$ is either a cyclic $p$-group or a generalized quaternion group $Q_{2^n}$ or $Q_{2^a}\times C_{q^b}$(where $a \geq 3, b\geq 1$ and $q$ is an odd prime) or a cyclic group $C_{p^{k}q^{s}}$, where $p, q$ are distinct primes and $k, s\geq 1$.
\end{theorem}

\begin{proof}
Let $G$ be a finite group such that $pqr|o(G)$, where $p, q, r$ are $3$ distinct primes. Then $G$ contains $3$ non-trivial subgroups say $H, K, L$ of order $p, q, r$ respectively. As intersection of any two subgroups is trivial so $\Gamma_{S}(G)$ contains a triangle with vertices $H, K, L$. Hence, $o(G)$ can have at most $2$ distinct prime divisors. \\
I) Let $o(G)=p^{\alpha}q^{\beta}$, where $\alpha, \beta \geq 1$. \\
a.) First suppose \textbf{both $p, q$ are odd}. If any one Sylow subgroup is non-cyclic (say Sylow $p$-subgroup is non-cyclic), then there exist two subgroups say $P_{1}, P_{2}$ of order power of $p$ such that they intersect trivially. Therefore, $\Gamma_{S}(G)$ contains a triangle with $3$ verices as $P_{1}, P_{2}, Q$, where $Q$ is the Sylow $q$-subgroup of $G$. Thus, we conclude that both Sylow subgroups are cyclic. Hence, $G$ is the cyclic group $C_{p^{k}q^{s}}$, where $p, q$ are distinct primes.\\
Now, let $o(G)=2^{\alpha}q^{\beta}$, where $\alpha, \beta \geq 1$ and $q$ is an odd prime. Clearly, the Sylow $q$-subgroup of $G$ must be cyclic; otherwise by the argument in I) a.) we can easily show that $\Gamma_{S}(G)$ carries a triangle. So, the Sylow $q$-subgroup is $C_{q^{\beta}}$.\\
Again, if the Sylow $2$-subgroup is cyclic then $G\cong C_{2^{\alpha}q^{\beta}}$, where $\alpha, \beta \geq 1$. \\
Let the Sylow $2$-subgroup is generalized quaternion group $Q_{2^{\alpha}}$(for $\alpha \geq 3$). We obtain $G \cong Q_{2^{\alpha}}\times C_{p^{\beta}}( \alpha \geq 3, \beta \geq 1)$.\\
But if the Sylow $2$-subgroup of $G$ is neither cyclic nor $Q_{2^{\alpha}}$, then just by the same argument as done in I) a.) we can prove that $\Gamma_{S}(G)$ comprises a triangle.\\
Therefore, if $o(G)$ has two distinct prime divisors then $G$ is either a cyclic group $C_{p^{\alpha}q^{\beta}}$ or $Q_{2^{\alpha}}\times C_{q^{\beta}}$.\\
II) Next consider that $G$ is a $p$-group. Here we consider the cases $p$ is odd prime and $p=2$.\\
\textbf{Case 1.} $p$ is an odd prime:\\
If $G$ is cyclic, then $\Gamma_{S}(G)$ is an empty graph; hence $\Gamma_{S}(G)$ is triangle-free.\\
Suppose $G$ is non-cyclic, so by Kulakoff Theorem $G$ has at least $3$ distinct subgroups of order $p$ that intersect trivially. Thus, $\Gamma_{S}(G)$ contains a triangle.\\
\textbf{Case 2.} $p=2$:\\
If $G$ is cyclic, just as in Case-1 clearly $\Gamma_{S}(G)$ is triangle-free.\\
Now, $G$ is other than generalized quaternion group then $G$ has at least $3$ subgroups of order $2$; these $3$ vertices forms a triangle in $\Gamma_{S}(G)$.\\
In case of generalized quaternion group, $\Gamma_{S}(G)$ is an empty graph as $G$ contains unique involution that belongs to each of the non-trivial subgroups of $G$. Hence in this case $G$ is triangle-free.\\
\textbf{For the converse part:}\\
Let $G$ is the cyclic group $C_{p^{k}q^{s}}$, where $p, q$ are distinct primes. The non-trivial subgroups are of order either a) power of $p$, or b) power of $q$, or c) of the form $p^{i}q^{j}$, where $1\leq i \leq k$, $1\leq j\leq s$ and $(i, j)\neq (k, s)$. If possible let $\Gamma_{S}(G)$ contains a triangle with vertices say $P, Q, R$. Now, if we select any of the options in (a), (b), or (c) for the orders of $P, Q$ and $R$, we reach a contradiction. Hence, $\Gamma_{S}(G)$ is triangle free.\\
Let $G\cong Q_{2^{\alpha}}\times C_{q^{\beta}}$. Then any two non-trivial subgroups of $Q_{2^{\alpha}}$ has order power of $2$ and they have non-trivial intersection. So, in a triangle there exists at most one subgroup of $Q_{2^{\alpha}}$. Similarly, as the Sylow $q$-subgroup is cyclic so only one vertex of a triangle be the subgroup of $C_{q^{\beta}}$. However, for any choices of the third vertex, we obtain a non-trivial intersection of the third vertex with the preceding vertices. Consequently, they are unable to form a triangle.\\
\end{proof}

\begin{corollary}
\label{cor_nil_bip}
Let $G$ be a finite nilpotent group. Then $\Gamma_{S}(G)$ is bipartite if and only if $G$ is either a cyclic $p$-group or a generalized quaternion group $Q_{2^n}$ or $Q_{2^a}\times C_{q^b}$ (where $a \geq 3, b\geq 1$ and $q$ is an odd prime) or a cyclic group $C_{p^{k}q^{s}}$, where $p, q$ are distinct primes and $k, s\geq 1$.
\end{corollary}

\begin{proof}
Let $G$ be a finite nilpotent group such that $\Gamma_{S}(G)$ is bipartite. Then $\Gamma_{S}(G)$ does not contain any odd cycle of length $3$ and above. This implies that $\Gamma_{S}(G)$ must be triangle-free. Hence, $G$ is one of the above mentioned group by the Theorem \ref{th_nil_triangle}.\\
Converse Part: If $G$ is any one of the stated group then $\Gamma_{S}(G)$ is perfect (according to the Theorem \ref{th_nil_perfect}) as well as triangle free (due to Theorem \ref{th_nil_triangle}). Thus, $\Gamma_{S}(G)$ does not contain an odd cycle of length $3$ and above. This implies $\Gamma_{S}(G)$ must be a bipartite graph.
\end{proof}

The above Theorem \ref{th_nil_triangle} characterizes the all nilpotent groups such that $\Gamma_{S}(G)$ is triangle-free. If we consider $G$ as any finite group, we observe that $G$ is either a cyclic $p$-group or a generalized quaternion group or $o(G)$ has exactly two distinct prime divisors such that either both the Sylow subgroups are cyclic or one of the Sylow subgroups is generalized quaternion group $Q_{2^n}$ and the other is a cyclic group. But we are unable to find the expicit structures of the group $G$. So, there is an open question that :
\begin{problem}
Classify the all non-nilpotent groups such that $\Gamma_{S}(G)$ is triangle-free.
\end{problem}

\section{Cluster graph}
A graph is called a cluster graph if it does not contain $P_{3}$ as its induced subgraph. \\
We are going to analyze all the finite nilpotent groups $G$ for which $\Gamma_{S}(G)$ is a cluster graph. 
\begin{theorem}
Let $G$ be a finite nilpotent group. Then $\Gamma_{S}(G)$ is a cluster graph if and only if $G$ is either a cyclic $p$-group(for any prime $p$) or generalized quaternion $2$-group or a $2$-generated $p$-group of exponent $p$ or $C_{pq}$, where $p, q$ are distinct primes.
\end{theorem}

\begin{proof}
Let $|\pi(G)|\geq 3$ and $G$ be a finite nilpotent group such that $pqr|o(G)$. Then $G$ contains a subgroup say $H$ of order $pq$, a subgroup $P$ of order $p$ and $R$ of order $r$. Thus $\Gamma_{S}(G)$ contains an induced path $(P, R, H)$, where $P\subset H$. Therefore $|\pi(G)|\leq 2$.\\
\textbf{Case 1.} $|\pi(G)|=2$: \\ Let us consider $G$ be a group of order $p^{\alpha}q^{\beta}$, where $\alpha, \beta \geq 1$. Let $P, Q$ be Sylow $p$ and $q$-subgroup of $G$ respectively.\\
\textbf{Subcase a.} Firstly, suppose both $p, q$ are odd primes.\\
We claim that both Sylow subgroups must be cyclic. If possible let, the Sylow $p$-subgroup of $G$ is non-cyclic. Then there exist two elements say $a, b$ of order $p$ and an element say $c$ of order $q$ such that $\Gamma_{S}(G)$ contains a path $(H_{1}, H_{2}, H_{3})$, where $H_{1}=\langle c\rangle, H_{2}=\langle a\rangle, H_{3}=\langle bc \rangle$. This gives us $G\cong C_{p^{\alpha}q^{\beta}}$. For the sake of contradiction we assume that at least one of $\alpha, \beta$ is larger than $1$. Let $\alpha >1$. Then there exist elements say $x, y\in G$ with $o(x)=p^{2}$ and $o(y)=q$ such that $(x, y, x^{p})$ is contained in $\Gamma_{S}(G)$. Therefore both $\alpha=\beta=1$ and hence $G\cong C_{pq}$.\\
\textbf{Subcase b.} Next, let $p=2$ and $q$ be an odd prime. \\
Now, $G\cong P'\times Q$, where $P', Q$ are Sylow $2$ and $q$-subgroup of $G$ respectively. \\
We assert that $G$ must be cyclic. \\
Clearly, we observe that the Sylow $q$-subgroup must be cyclic; otherwise we get $P_{3}$. Now, if the Sylow $2$-subgroup $P'$ is generalized quaternion group then $P'$ has two subgroups of order $4$ say $A, B$ with $A\cap B\neq \{e\}$. Hence, $\Gamma_{S}(G)$ contains a path $A, Q, B$.\\
On the other hand, if $P'$ is neither cyclic nor generalized quaternion then just as the similar argument as done in subcase-a we can easily show that $\Gamma_{S}(G)$ contains a path $P_{3}$. Thus, $G$ must be the cyclic group $C_{2q}$ (c.f. the argument as done in Subcase-a). \\
Thus, if $|\pi(G)|=2$ then $G \cong C_{pq}$, where $p, q$ are distinct primes.\\
\textbf{Case 2.} $|\pi(G)|$=1. \\ Then $G$ is a $p$-group. If $G$ is cyclic or generalized quaternion $2$-group then $\Gamma_{S}(G)$ is an empty graph and hence it is cluster graph.\\
Let $G$ is neither cyclic nor a generalized quaternion $2$-group. Then there exist two elements say $a, b$ of order $p$. If $G$ has an element of order greater than $p$ say $x$ such that $x^{p}=a$ then $\Gamma_{S}(G)$ contains a $3$-vertex induced path $(\langle x\rangle , \langle b\rangle , \langle a \rangle)$. Thus, $G$ must be a $p$-group of exponent $p$.\\
Moreover, we observe that $G$ must be $2$-generated; elsewhere $G$ contains a subgroup $C_{p}\times C_{p}\times C_{p}$. One can check $\Gamma_{S}(C_{p}\times C_{p}\times C_{p})$ contains $P_{3}$.\\
\textbf{Converse Part:} Let $G$ is a $2$-generated $p$-group of exponent $p$. Then every non-identity element is of order $p$; thus any $3$-vertices form a triangle. Hence, $\Gamma_{S}(G)$ is a cluster graph.\\
Again, let $G\cong C_{pq}$. Then $G$ has unique non-trivial subgroup is of order $p, q$. So, $\Gamma_{S}(G)$ is unable to contain a $3$-vertex induced path. 
\end{proof}

%%%%%%%%%%%%%%%%%%%%%%%%%%%%%%%%%%%%%%%%%%%%%%%%%%%%%%%%

\section{Claw-free}
A graph is called claw-free if it does not contain $K_{1, 3}$ as its induced subgraph.\\
Claw-free graph are the subclass of line graphs. In this sction we investigate the finite nilpotent groups $G$ such that $\Gamma_{S}(G)$ is claw-free.
\begin{theorem}
Let $G$ be a finite nilpotent group. Then $\Gamma_{S}(G)$ is claw-free if and only $G$ is either cyclic $p$-group or generalized quaternion $2$-group or a $3$-generated $p$-group of exponent $p$ or a $2$-generated $p$-group in which every non-identity elements are of order $p, p^2$ or a cyclic group $C_{p^{a}q^{b}}$ (where $p, q$ are primes and $a, b \geq 1$) or $P\times C_{q}$, where $P$ is a non-cyclic, $2$-generated $p$-group of exponent $p$ or a cyclic group $C_{pqr}$ ($p, q, r$ are $3$ distinct primes).
\end{theorem}

\begin{proof}
Let $G$ be a finite nilpotent group such that $\Gamma_{S}(G)$ is claw-free. We observe that, $|\pi(G)|\leq 3$; elsewhere $\Gamma_{S}(G)$ contains a claw with $3$ pendent vertices are of orders $q, qr, qs (r \neq s)$ and the central vertex is of order $p$.\\
\textbf{Case 1.} $|\pi(G)|=3$\\
Let $p<q<r$ be distinct primes dividing $o(G)$ and $P, Q, R$ be the corresponding Sylow subgroups of $G$ respectively. We claim that the Sylow subgroups of $G$ must be cyclic. \\
First, we show that Sylow $q$ and $r$ subgroup must be cyclic. Otherwise, if Sylow $q$-subgroup is non-cyclic then $G$ has at least two subgroups say $H, K$ of order $q$. So, $\Gamma_{S}(G)$ contains a claw with $3$ pendent vertices as $A, B, C$ (where, $H \subset C$) of orders $p, pr, pq$ respectively and central vertex is $K$. Similarly, we can say that the Sylow $r$-subgroup must be cyclic. Therefore, $G\cong P\times C_{q^{k}r^{l}}$.\\
Next, we observe that $k=l=1$; elsewhere $\Gamma_{S}(G)$ carries a claw. \\
Now, let $p$ be odd. So, just as similar argument for Sylow $q$-subgroup we can say $P$ must be cyclic group of order $C_{p}$.\\
If $p=2$, then $P$ is either $C_{2}$ or $Q_{2^k}$. As otherwise $P$ have two distinct subgroups of order $p$ so we can construct a claw. \\
If $P\cong Q_{2^k}$, then $P$ has two subgroups say $H, K$ of order $4$. This gives that $\Gamma_{S}(G)$ contains a claw together with $3$ pendent vertices as $H, K, A$ and the central vertex $Q$, where $R, H \subset A$. Hence, $P\cong C_{2}$.\\
Therefore, in any aspect $G\cong C_{pqr}$, where $p, q, r$ are $3$ distinct primes.\\
\textbf{Case 2.} $|\pi(G)|=2$\\
Let $o(G)=p^{a}q^{b}$. Here we consider two cases depending on $p, q$.\\
\textbf{Subcase 1.} both $p, q$ are odd primes\\
Let $P, Q$ be the Sylow subgroups of $G$. Our assertion is that one of two Sylow subgroups must be cyclic. Due to potential contradiction, we suggest that both $P, Q$ are non-cyclic. Then $G$ contains two $p$-subgroups $P_{1}, P_{2}$ and two $q$-subgroups $Q_{1}, Q_{2}$. Therefore, $\Gamma_{S}(G)$ contains a claw with $3$ pendent vertices $P_{1}\times Q_{1}, P_{1}, P_{1}\times Q_{2}$ and the central vertex as $P_{2}$. This proves that one Sylow subgroup say $Q$ must be cyclic. Then $G\cong P\times C_{q^b}$( where, $r\geq 1$).\\
If $P$ is cyclic then $G \cong C_{p^{a}q^{b}}$( $a, b\geq 1$).\\
Let $P$ is non-cyclic. Our claim is that $P$ must be a $2$-generated $p$-group of exponent $p$. Otherwise, $G$ contains a subgroup $M \cong C_{p^k}\times C_{p}\times C_{p}\times C_{q^b}$. But, $\Gamma_{S}(M)$ contains a claw. This gives $G\cong P\times C_{q^b}$ such that $P$ is a non-cyclic, $2$-generated $p$-group of exponent $p$. Moreover, if $b >1$ then $G$ has a subgroup $C_{p}\times C_{p}\times C_{q^b}$, whose intersection subgroup graph contains a claw. Thus $G \cong P\times C_{q}$, where $P$ is a non-cyclic, $2$-generated $p$-group of exponent $p$.\\
\textbf{Subcase 2.} one of $p, q$ is $2$\\
 Let $o(G)=2^{a}q^{b}$. Using similar argument as done in Subcase-1 we conclude that the Sylow $q$ subgroup must be cyclic.
Let $P$ be the Sylow $2$-subgroup of $G$. \\
If $P$ is cyclic then $G\cong C_{2^{a}q^{b}}$.\\
Let $P$ be generalized quaternion. Then $G$ comprises a subgroup $C_{4}\times C_{4}\times C_{q^b}$. As, $\Gamma_{S}(C_{4}\times C_{4}\times C_{q^b})$ contains a claw so is $\Gamma_{S}(G)$.\\
Now, let $P$ be neither cyclic nor a generalized quaternion. Then by using similar argument as done in Subcase-1, we conclude that $P$ must be a non-cyclic, $2$-generated $2$-group of exponent $2$.\\
Thus, in Case-2 we obtain $G$ is either a cyclic group $C_{p^{a}q^{b}}$ or $P\times C_{q^b}$, where $P$ is a non-cyclic, $2$-generated $p$-group of exponent $p$.\\
\textbf{Case 3.} $|\pi(G)|=1$.\\
 Let $o(G)=p^{a}$. \\
If $G$ is either cyclic or generalized quaternion $2$-group then it is clear that $\Gamma_{S}(G)$ is claw-free.\\
For the other $p$-groups we claim that $G$ can have elements of orders $p$ or $p^{2}$. Let $G$ have an element of order $p^{k}$ where $k\geq 3$, then $\Gamma_{S}(G)$ contains a claw $A, B, C, D$, where $A, B, C \in C_{p^k}$ and $D$ has order $p$ such that the intersection of $D$ with any one of $A, B, C$ is trivial. \\
Now, let $G$ have an element of order $p^2$. Then $G$ must be $2$-generated; elsewhere $G$ contains $C_{p^k}\times C_{p}\times C_{p}$. One can check $\Gamma_{S}(G)$ contains a claw.\\
Again, if $G$ is a $p$-group of exponent $p$ then $G$ must be $3$-generated; otherwise as $G$ has a subgroup $(C_{p})^{k}$, where $k \geq 4$. But $\Gamma_{S}(G)$ contains a claw.\\
\textbf{Converse part }:
If $G$ is cyclic $p$-group or generalized quaternion $2$-group, then there is nothing to prove. \\
Let $G$ be a non-cyclic, $3$-generated $p$ group of exponent $p$. Clearly, $G$ has $3$ distinct subgroups say $P_{1}, P_{2}, P_{3}$ of orders $p$. If $\Gamma_{S}(G)$ contains a claw then any two pendent vertices must belong to distinct cyclic $p$-subgroup. Thus they are adjacent, which implies that $\Gamma_{S}(G)$ is claw-free.\\
Similarly, for the case $G$ is a $2$-generated $p$-group in which every element is of order $p$ or $p^2$ one can prove that $\Gamma_{S}(G)$ is claw-free. \\
On the other, hand if $G$ is the cyclic group $C_{p^{a}q^{b}}$, where $a, b \geq 1$, then as it contains unique subgroup of each order so we can choose $3$ vertices of order power of $p$, power of $q$. But the remaining vertex must be of order divisible by $pq$ whose intersection is non-trivial with the other vertices. Hence $G$ must be claw-free in this case.\\
Again, let $G\cong P\times C_{q}$, where $P$ is a $2$-generated $p$-group of exponent $p$. Then $P$ contains two subgroups, say $P_{1}, P_{2}$, of orders $p$. If possible let, $\Gamma_{S}(G)$ contain a claw. Without loss of generality, let the central vertex be $P_{2}$. Then the pendent vertex can be $P_{1}, P_{1}\times C_{q}, C_{q}$. But this fails to form a claw as $P_{1}\cap C_{q}=\{e\}$.\\
Also, if $G\cong C_{pqr}$, it is easy to prove that $\Gamma_{S}(G)$ is claw-free.
\end{proof}

\begin{problem}
Find the all finite nilpotent groups $G$ such that $\Gamma_{S}(G)$ is a line graph.
\end{problem}

%%%%%%%%%%%%%%%%%%%%%%%%%%%%%%%%%%%%%%%%%%%%%%%%%%%%%%%%

\section{Cograph}
A graph $G$ is a cograph if it has no induced subgraph isomorphic to the four-vertex path $P_4$. Cographs form the smallest class of graphs 
containing the $1$-vertex graph and closed under the operations of disjoint
union and complementation.\\
This section is devoted to the finite nilpotent groups, the symmetric group $S_{n}$,  the alternating group $A_{n}$ and the dihedral groups $D_{n}$ whose intersection subgroup graph is a chordal
graph.
\begin{proposition}
\label{lm_p_co}
Let $G$ be a $p$-group. Then $\Gamma_{S}(G)$ is a cograph if and only if $G$ is either a cyclic $p$-group or a generalized quaternion $2$-group or a non-cyclic, $2$-generated $p$-group.
\end{proposition}

\begin{proof}
Let $G$ be a $p$-group such that $\Gamma_{S}(G)$ is a cograph.
To prove the theorem we consider two cases depending on $p$.\\
\textbf{Case 1.} $p$ is an odd prime\\
If $G$ is cyclic then clearly $\Gamma_{S}(G)$ is a cograph.\\
Next suppose, $G$ is non-cyclic. We claim that $G$ is a $2$-generated $p$-group. For the sake of contradiction, we assume that $G$ contains a subgroup of the form $(C_{p})^{3}$; then $\Gamma_{S}(G)$ contains a path $C_{p}\times C_{p} \times \{e\}, \{e\}\times \{e\}\times C_{p}, C_{p}\times \{e\}\times \{e\}, \{e\}\times C_{p}\times C_{p}$. Hence, we conclude that $G$ must be $2$-generated.\\
%Then there exist two elements say $a, b$ of order $p$. So, $G$ has at least two subgroups say $H, K$ of order power of $p$ such that $\langle a \rangle \in H$ and $\langle b \rangle \in K$. Moreover, $H, K$ satisfies the property that $\langle b \rangle \cap H=\{e\}$ and $\langle a \rangle \cap K=\{e\}$. If possible let, $\Gamma_{S}(G)$ contains a path in which first two vertices are $\langle a \rangle$ and $\langle b \rangle$. This implies the next vertex must be $H$ and the final vertex must be $K$; but $K$ is adjacent to the first vertex. \\
%Next we choose a $4$-vertex path of $\Gamma_{S}(G)$ in which $3$ vertex belong to $3$ distinct subgroups that intersects to each other trivially. Then they forms a triangle and hence such path never exist. Therefore, $\Gamma_{S}(G)$ is a cograph.\\
\textbf{Case 2.} $p=2$\\
If $G$ is either cyclic or generalized quaternion, then there is nothing to prove.\\
Now, let $G$ is neither cyclic nor generalized quaternion group. Then using the above argument as done in Case-1, we can conclude that $G$ must be a $2$-generated $2$-group.\\
\textbf{Converse Part:}\\
Let $G$ be a non-cyclic, $2$-generated $p$-group. We prove that $\Gamma_{S}(G)$ is a cograph. On the contrary, we suppose $\Gamma_{S}(G)$ contains a $4$-vertex induced path say $H_{1}, H_{2}, H_{3}, H_{4}$. Then each $H_{i}$ has order power of $p$ and either they are cyclic or $2$-generated. If at least one $H_{i}$ is non-cyclic then its intersection with the other vertices must be non-trivial (since, $C_{p}$ belongs to the intersection of any two $H_{i}$). Thus, each $H_{i}$ must be cyclic. In that case, $H_{1}, H_{3}$ belong to the same cyclic subgroup. On the other hand, $H_{4} \nsim H_{2}$ implies they must lie in the same cyclic subgroup. But, as $H_{2}, H_{4}\cap H_{3}=\{e\}$ so $H_{4}$ intersects $H_{1}$ trivially. Then $H_{1}\sim H_{4}$. Therefore, there does not exist any $4$-vertex induced path in $\Gamma_{S}(G)$, which implies that $\Gamma_{S}(G)$ must be a cograph in this case.
\end{proof}

\begin{theorem}
\label{th_nil_co}
Let $G$ be a finite nilpotent group which is not a $p$-group. Then $\Gamma_{S}(G)$ is a cograph if and only if  $G$ is either a cyclic group $C_{p^{a}q^{b}}$, where $a, b \geq 1$, or of the form $Q_{2^a}\times C_{q^b}$ ($a \geq 3$ and $b \geq 1$).
\end{theorem}

\begin{proof}
Let $G$ be a finite nilpotent group, which is not a $p$-group, such that $\Gamma_{S}(G)$ is a cograph.\\
Firstly, let $|\pi(G)|\geq 3$. Then $\Gamma_{S}(G)$ contains a $4$ vertex induced path where four vertices are $(H_{1}, H_{2}, H_{3}, H_{4})$ of orders $pr, q, p, qr$ respectively. Thus $G$ can have exactly two distinct prime divisors. \\
Let $o(G)=p^{a}q^{b}$, where $a, b \geq 1$ and $P, Q$ be the Sylow $p$ and $q$-subgroups of $G$ respectively.\\
Here we consider two cases based on $p, q$.\\
\textbf{Case 1.}both $p, q$ odd primes\\ In this case, we claim that both the Sylow subgroups of $G$ must be cyclic; otherwise if any one of $P, Q$ is non-cyclic (say $P$) then there exist two elements say $a, b$ of order $p$ and an element $c$ of order $q$ in $G$. So, $\Gamma_{S}(G)$ contains a path $(\langle ac \rangle, \langle b \rangle, \langle a \rangle, \langle bc \rangle)$. Therefore $G$ is a cyclic group $C_{p^{a}q^{b}}$, where $a, b \geq 1$.\\
\textbf{Case 2.} $p=2$ and $q$ is an odd prime\\
Clearly, the Sylow $q$-subgroup $Q$ of $G$ must be cyclic; elsewhere $\Gamma_{S}(G)$ comprises a $4$-vertex induced path. Thus, $G\cong P\times C_{q^b}$ ($b \geq 1$).\\
We claim that, $P$ is either cyclic or generalized quaternion. If possible let, $P$ be neither cyclic nor generalized quaternion group. Then $P$ has at least two distinct subgroups of order $2$. So, by the same argument as done in Case-1 we can prove that $\Gamma_{S}(G)$ contains a $4$-vertex induced path.\\
If $P$ is cyclic then $G\cong C_{2^{a}q^{b}}$.\\
Let $P$ be the generalized quaternion group $Q_{2^a}$. Then $G\cong Q_{2^a}\times C_{q^b}$, where $a\geq 3, b\geq 1$.\\
\textbf{Converse Part:} Let $G$ be the cyclic group $C_{p^{a}q^{b}}$, where $a, b \geq 1$. For the purpose of contradiction, we assume that $\Gamma_{S}(G)$ contains a $4$-vertex path say $(A, B, C, D)$. Without loss of generality, let $o(A), o(C)$ be power of $p$ and $o(B)$ be power of $q$. Then $o(D)$ must be power of $q$. Since, $G$ contains unique subgroup of each order so $A\sim D$. Hence in this case $\Gamma_{S}(G)$ must be a cograph.\\
Suppose, $G$ is the group  $Q_{2^a}\times C_{q^b}$, where $a\geq 3, b\geq 1$. We prove that $\Gamma_{S}(G)$ is a cograph. Contrarily, we assume that $\Gamma_{S}(G)$ carries a $4$-vertex induced path $A, B, C, D$. Clearly, any two subgroup of order $4$ or $4q^{k}$ or $2q^{r}$ has non-trivial intersection. Without loss of generality, let $A, C$ be two distinct subgroups of $G$ of order $4$ and $B$ be the subgroup $C_{q^b}$. So, $D$ must be $C_{q}$. But, then $A, C\sim D$. This implies that $\Gamma_{S}(G)$ does not contain any $4$-vetex induced path. On the other hand, if $b=1$ that is $G\cong Q_{2^a}\times C_{q}$ then $\Gamma_{S}(G)$ contains an induced path of length at most $2$. Thus, in any aspect , $\Gamma_{S}(G)$ is a cograph.
\end{proof}

\begin{theorem}
\label{th_sym}
$\Gamma_{S}(S_{n})$ is cograph if and only if $n=3$. Moreover, $\Gamma_{S}(A_{n})$ is cograph if and only if $n \leq 4$.
\end{theorem}

\begin{proof}
\textbf{First Part:(symmetric group on $n$-symbols ($S_{n}$))}\\
Let $n=4$. Now, $S_{4}$ has following subgroups:\\
a) $S_{3}=\langle (1, 2, 3), (1, 2) \rangle$; b) $D_{4}=\langle (1, 2, 3, 4), (1, 3) \rangle$; c) $(C_{2})^{2}=\langle (1, 3), (2, 4) \rangle$; d) $C_{2}=\langle (2, 4) \rangle$. So, $\Gamma_{S}(G)$ contains a $4$-vertex induced path $((C_{2})^{2}, S_{3}, D_{4}, C_{2})$. Hence, $\Gamma_{S}(S_{4})$ is not a cograph.\\
Therefore, if $n \geq 4$ then $\Gamma_{S}(S_{n})$ is not a cograph. Thus, $n=3$.\\
For the converse, let $n=3$. The only non-trivial subgroup of $S_{3}$ are: $A_{3}$ and $\langle (1, 2) \rangle$, $\langle (1, 3) \rangle$, $\langle (2, 3) \rangle$. But in this case $\Gamma_{S}(S_{3})$ forms a complete graph. Hence, $\Gamma_{S}(S_{3})$ is a cograph.\\
\textbf{Second Part:(alternating group on $n$ symbols($A_{n}$))}\\
We claim that $\Gamma_{S}(A_{n})$ is not a cograph for $n \geq 5$.\\
Let $n=5$. Then $\Gamma_{S}(A_{5})$ contains a path: $(D_{5}, A_{4}, C_{5}, S_{3})$, where $D_{5}=\langle (1, 2, 3, 4, 5), (2, 5)(3, 4) \rangle, A_{4}=\langle (1, 2)(3, 4), (1, 2, 3)\rangle, C_{5}=\langle (1, 2, 3, 4, 5) \rangle, S_{3}=\langle (1, 2, 3), (1, 2)(4, 5)\rangle$. Therefore, $\Gamma_{S}(A_{5})$ is not a cograph.\\
Hence, for $n \geq 5$, $\Gamma_{S}(A_{n})$ is not a cograph.\\
Thus $n \leq 4$.\\
Converse Part: If $n=3$, then $A_{3}$ is nothing but the cyclic group $C_{3}$. Since, $\Gamma_{S}(C_{3})$ is null graph so it is a cograph.\\
Next consider the group $A_{4}$. The non-trivial subgroups of $A_{4}$ are: $C_{2}\times C_{2}, C_{2}, C_{3}$. Now the subgroups $C_{2}$ share a non-trivial intersection with the unique subgroup $C_{2}\times C_{2}$. Also, any two $C_{3}$ has trivial intersection, similar holds for any two $C_{2}$. Therefore to form a path in $\Gamma_{S}(A_{4})$, the three consecutive vertices must be $C_{2}, C_{3}, C_{2}\times C_{2}$. So, if we add another vertex then it must be either $C_{2}$ or $C_{3}$; which is adjacent to the previously added vertex. Hence $\Gamma_{S}(A_{4})$ is a cograph.
\end{proof}

\begin{theorem}
\label{th_dih}
$\Gamma_{S}(D_{n})$ is a cograph if and only if $n$ is a prime power.
\end{theorem}

\begin{proof}
Firstly, suppose $\Gamma_{S}(D_{n})$ is a cograph.\\
If $n$ is divisible by $3$ or more primes, then $G$ contains a subgroup $C_{n}$. But $\Gamma_{S}(G)$ is not a cograph. Hence, $n$ can have at most $2$ distinct prime divisors.\\
If $n=p^{k}q^{r}$. Then $\Gamma_{S}(D_{n})$ contains a path $(\langle a^{q}, ab \rangle, \langle a^{p}, b \rangle, \langle a^{q} \rangle, \langle a^{p}, ab \rangle)$. Therefore, $n$ must be a prime power.\\
Conversely, let $n=p^{r}$(where, $p$ is a prime and $r \geq 1$).
Since, $D_{n}$ consists two types of subgroups: \\
Type-I: $\langle a^{d} \rangle$, where $d|n$;\\
Type-II: $\langle a^{d}, a^{i}b \rangle$, where $d|n$ and $0\leq i \leq (d-1)$.\\
Since, $n$ is power of $p$ so $d$ must be power of $p$ and then any two subgroups of Type-I and Type-II share an element of the form $a^{p^k}$ in common. Similarly, a subgroup of Type-I and Type-II share non-trivial intersection. Hence, any $3$ vertices of $\Gamma_{S}(D_{n})$ form a triangle. Thus, there exists no $4$-vertex induced path in $\Gamma_{S}(D_{n})$. This makes $\Gamma_{S}(D_{n})$ a cograph.
\end{proof}

%%%%%%%%%%%%%%%%%%%%%%%%%%%%%%%%%%%%%%%%%%%%%%%%%%%%%%%%

\section{Chordal Graph}
A graph $\Gamma$ is said to be chordal if it does not contain any induced cycles of length greater than $3$; in other words, every cycle on more than $3$ vertices has a chord.\\
In this section, we examine the finite nilpotent groups, solvable groups, the symmetric group $S_{n}$, the alternating group $A_{n}$ and the dihedral groups $D_{n}$ whose intersection subgroup graph is a chordal
graph.
\begin{proposition}
Let $G$ be a finite abelian $p$-group. Then $\Gamma_{S}(G)$ is chordal if and only if $G$ is either a cyclic $p$-group or $C_{p^k}\times C_{p}$ (where, $k \geq 1$) or $C_{p}\times C_{p}\times C_{p}$.
\end{proposition}

\begin{proof}
It is known that, if $G$ is a cyclic $p$-group then $\Gamma_{S}(G)$ is a null graph and hence chordal.\\
So, let $G$ be a finite abelian (non-cyclic) $p$-group such that $\Gamma_{S}(G)$ is chordal.\\
Let $G\cong C_{p^{k_1}}\times C_{p^{k_2}}\times \cdots \times C_{p^{k_r}}$. Let there exist two $k_{i}>1$ (for $1 \leq i \leq r$); say $k_{1}, k_{2}>1$. Then $\Gamma_{S}(G)$ contains a $4$-vertex cycle as $(A, B, C, D, A)$ where $A=C_{p^{k_1}}\times \{e\}\times \cdots \times \{e\}$, $B=\{e\}\times C_{p^{k_{2}}}\times \cdots \times \{e\}$, $C=C_{p}\times \{e\}\times \cdots \times \{e\}$, $D=\{e\}\times C_{p}\times \cdots \times \{e\}$. Hence, $G\cong C_{p^{k}}\times C_{p}\times \cdots \times C_{p}$, where $k \geq 1$.\\
If $G$ is the product of two copies of cyclic $p$-groups then $G$ must be $C_{p^k}\times 
C_{p}$.\\
Let $G$ is product of $3$ or more copies of cyclic $p$-groups. Then $G\cong C_{p^k}\times C_{p}\times C_{p}\times \cdots \times C_{p}$. If $k>1$ then $\Gamma_{S}(G)$ contains a cycle $(A, B, C, D, A)$ where $A=C_{p^k}\times \{e\}\cdots \times \{e\}, B=\{e\}\times C_{p}\times \cdots \times C_{p}, C=C_{p}\times \{e\}\times \cdots \times \{e\}, D=\{e\}\times C_{p}\times \{e\}\times \cdots \times \{e\}$. Hence, $G$ is $(C_{p})^{r}$, where $r \geq 3$.\\
In a similar manner, we can show that if $r \geq 4$ then $\Gamma_{S}(G)$ contains a cycle of length $4$. Therefore, $r=3$ and $G\cong C_{p}\times C_{p}\times C_{p}$.\\
Conversely, 
let $G$ be $C_{p^k}\times C_{p}$. If possible let, $\Gamma_{S}(G)$ contain a cycle $(A, B, C, D,..., A)$. Without loss of generality, let $A=C_{p^k}\times \{e\}, B=\{e\}\times C_{p}, C=C_{p}\times \{e\}$. Then for any choices of $D$, it shares a non-trivial intersection with $A, C$. So such cycle does not exist. Hence, $\Gamma_{S}(G)$ is chordal.\\
Let $G\cong C_{p}\times C_{p}\times C_{p}$. Suppose $\Gamma_{S}(G)$ contains a cycle $(A, B, C, D, A)$. Without loss of generality we take $A, B, C$ as $C_{p}\times \{e\}\times C_{p}, \{e\}\times C_{p}\times \{e\}, \{e\}\times \{e\}\times C_{p}$. Then $D$ will be $C_{p}\times C_{p}\times \{e\}$, which is not adjacent to $A$. Therefore, such $4$-cycle does not exist. Similarly, we can show that $\Gamma_{S}(G)$ does not contain any cycle of length $5$ and above. Hence, $\Gamma_{S}(G)$ is chordal
\end{proof}

\begin{proposition}
Let $G$ be a finite non-abelian $p$-group such that $\Gamma_{S}(G)$ is chordal. Then $G$ has the following possibilities:\\
a) a $2$-generated $p$-group such that $G$ is either generalized quaternion or $G$ does not contain any subgroup of the form $C_{p^k}\times C_{p^r}$ where $k, r>1$;\\
b) a $3$-generated $p$-group that does not contain any subgroup of the form $C_{p^k}\times C_{p}\times C_{p}$, where $k>1$.
\end{proposition}

\begin{proof}
Let $G$ be a non-abelian $p$-group such that $\Gamma_{S}(G)$ is chordal.\\
If $G$ is a generalized quaternion group, then there is nothing to prove.\\
Now, let $G$ is not a generalized quaternion group.\\
We claim that $G$ cannot be generated by $4$ and above elements.\\
Suppose, $G$ contains a subgroup of the form $(C_{p})^{4}$. Then $\Gamma_{S}(G)$ contains a cycle $C_{p}\times \{e\}\times C_{p} \times\{e\}, \{e\}\times C_{p}\times \{e\}\times \{e\}, C_{p}\times \{e\}\times \{e\}\times \{e\}, \{e\}\times C_{p}\times C_{p}\times \{e\}\times C_{p}\times \{e\}\times C_{p} \times\{e\}$. Hence, we reach the conclusion that $G$ is either a $2$-generated or a $3$-generated $p$-group.\\
Let, $G$ be a $2$-generated $p$-group.\\
It is clear that if $G$ contains a subgroup of the form $C_{p^k}\times (C_{p})^{r}$, where $r \geq 2$, then $\Gamma_{S}(G)$ contains a $4$-vertex cycle. Hence we conclude $G$ is a group of the form a).\\
Next suppose, $G$ is $3$-generated. If $G$ contains a subgroup of the form $C_{p^k}\times C_{p}\times C_{p}$. Then $\Gamma_{S}(G)$ contains a $4$-vertex cycle. Thus we get the conclusion b).
\end{proof}

In the above discussions we are able to characterize the all finite abelian $p$-groups such that $\Gamma_{S}(G)$ is chordal. But if $G$ is a non-abelian $p$-group we explore the structural properties of the groups such that $\Gamma_{S}(G)$ is chordal. Based on that, we pose a problem:
\begin{problem}
Characterize all the $p$-groups such that $\Gamma_{S}(G)$ is chordal.
\end{problem}

\begin{theorem}
Let $G$ be a finite nilpotent group which is other than $p$-group. Then $\Gamma_{S}(G)$ is chordal if and only if $G$ is one of the following:\\
a)If $|\pi(G)|=2$, $G$ is either a cyclic group $C_{{p^k}q}$ for $k \geq 1$ or $Q_{2^a}\times C_{q}$ ($a\geq 3$) or $G$ is non-cyclic group with one Sylow subgroup is cyclic and other is a $2$-generated group of prime exponent;\\
b)If $|\pi(G)|=3$, $G\cong C_{pqr}$.
\end{theorem}

\begin{proof}
Suppose, $G$ is a finite nilpotent group such that $|\pi(G)| \geq 4$. First we observe that if $pqrs|o(G)$ (where $p<q<r<s$) then $\Gamma_{S}(G)$ contains an induced $4$-cycle with vertices say $(H_{1}, H_{2}, H_{3}, H_{4}, H_{1})$, where $o(H_{1})=p, o(H_{2})=r ,o(H_{3})= ps, o(H_{4})=qr $. Therefore, $|\pi(G)|\leq 3$.\\
\textbf{Case 1.} $|\pi(G)|=3$\\
Clearly, $o(G)=p^{l}q^{m}r^{n}$. If any one of $l, m, n$ is greater than $1$, then $\Gamma_{S}(G)$ contains an induced $4$-cycle: $(A, B, C, D, A)$ with $o(A)=p^{l}, o(B)=q, o(C)=p, o(D)=qr$ and $C\subset A$. Thus, $o(G)=pqr$ and hence $G\cong C_{pqr}$.\\
\textbf{Case 2.} $|\pi(G)|=2$\\
Let $o(G)=p^{k}q^{r}$. Consider the Sylow $p$ and $q$-subgroups of $G$ as $P, Q$ respectively. We claim that at least one of $k, r$ is $1$; elsewhere there exists a $4$-cycle with vertices as $(P, A, B, Q, P)$ where $B\subset P$ and $A\subset Q$. Let $r=1$ and hence $o(G)=p^{k}q$. Our next claim is that $P$ is either cyclic or generalized quaternion or a non-cyclic $p$-group of exponent $p$. \\
If $P$ is cyclic then $G$ is cyclic.\\
If $P$ is generalized quaternion group $Q_{2^a}$, where $a \geq 3$, then $G \cong Q_{2^a}\times C_{q}$.\\
Let $P$ be neither cyclic nor generalized quaternion. Then $G$ contains two distinct elements say $a, b$ of order $p$. If $P$ contains an element say $x$ of order more than $p$, then 
there exist two subgroups say $H, K$ in $G$ such that $H=\langle x \rangle, K=\langle x^{p}\rangle$ (where, $a=x^{p}$). In this case $\Gamma_{S}(G)$ contains a cycle $K,  L, H, Q, K$ where $L=\langle bc \rangle$ and $Q=\langle c \rangle$. Hence $P$ must be a $p$-group of exponent $p$. \\
Next we observe that $P$ must be $2$-generated group; elsewhere $G$ will contains a subgroup $G'=C_{p}\times C_{p}\times C_{p}\times C_{q}$, such that $\Gamma_{S}(G')$ contains a $4$-vertex induced cycle. Hence in this case $\Gamma_{S}(G)$ is not chordal.\\
\textbf{Converse part:}\\  i) Let $G\cong C_{p^{k}q}$, where $k \geq 1$. Now we prove that $\Gamma_{S}(G)$ is chordal. \\
If not, let $\Gamma_{S}(G)$ contains a cycle of length $4$ and above. It is clear that every non-trivial subgroup is of order either power of $p$, or of the form $p^{i}q$ where $0\leq i \leq k-1$. Since $G$ has unique subgroup of each order, so if we try to construct a cycle of length $4$ and above then there must exist at least two subgroups of order power of $p$ and two subgroups of order of the form $p^{i}q$. But then we are unable to make a cycle. Hence in this case $\Gamma_{S}(G)$ is chordal.\\
ii) Let $o(G)=p^{k}q=P\times Q$ such that the Sylow $p$-subgroup $P$ of $G$ is a $2$-generated $p$-group of exponent $p$ and the Sylow $q$-subgroup $Q$ is cyclic. If possible let, $\Gamma_{S}(G)$ contains a cycle of length $4$ and above. Since $P$ is non-cyclic, then $P$ has at least two distinct elements say $a, b$ of order $p$. Consider the subgroups $P_{1}=\langle a \rangle, P_{2}=\langle b \rangle$. Without loss of generality we choose three vertices of a cycle as $P_{2}\times Q, P_{1}, Q$ then the next vertex will be a subgroup must contain $P_{1}$. But as $P$ is non-cyclic and of exponent $p$ so any subgroup of order more than $p$ must contain both $P_{1}, P_{2}$. Thus the only possibility for the next vertex is $P$. So it cannot be adjacent to $P_{2}\times Q$. Similarly if we check for higher length we cannot form a cycle. Hence in this case $\Gamma_{S}(G)$ is chordal.\\
iii) Let, $G\cong Q_{2^a}\times C_{q}$. Now by Theorem \ref{th_nil_co}, $\Gamma_{S}(G)$ is cograph. So, for the chordality we have to check only the existence of a $4$-cycle in $\Gamma_{S}(G)$. Suppose that $\Gamma_{S}(G)$ carries a $4$-vertex induced cycle $A, B, C, D, A$. Clearly, any two subgroups of order $4$ or $4q$ or $2q$ share a non-trivial intersection. Without loss of generality, let $A, C$ be two distinct subgroups of $G$ of order $4$ and $B$ be the subgroup $C_{q}$. Then, for any choices of $D$ we must have either $D\nsim C, A$ or $D\cong B$. Hence, such $4$-cycle does not exist in $\Gamma_{S}(G)$, which proves that $\Gamma_{S}(G)$ is chordal.\\
iv) If $G\cong C_{pqr}$ then any non-trivial subgroup posseses orders $\{p, q, r, pr, pq, qr\}$. Also, $G$ contains unique subgroup of each order. Without loss of generality, we choose consecutive $4$ vertices of a cycle as $A, B, C, D$ of respective orders $qr, p, q, pr$. So if we add another vertex, it must be of order $pq$ or $r$. But for any of these choices, that vertex is not adjacent to $D$. Thus $\Gamma_{S}(G)$ does not contain a cycle of length $5$ and above. Similarly it does not contain any cycle of length $4$. Hence, $\Gamma_{S}(G)$ is chordal.
\end{proof}

\begin{theorem}
Let $G$ be a solvable group with $|\pi(G)|\geq 2$. Then $\Gamma_{S}(G)$ is chordal graph if $G$ is either a group of order $p_{1}p_{2}p_{3}$, where $p_{1}<p_{2}<p_{3}$ or $G$ has order $p_{1}^{k}p_{2}$, where $k\geq1$.
\end{theorem}

\begin{proof}
First we consider a solvable group $G$ with $|\pi(G)|\geq 2$ such that $\Gamma_{S}(G)$ is chordal.\\
Let $o(G)=p_{1}^{\alpha_{1}}p_{2}^{\alpha_{2}}\cdots p_{r}^{\alpha_{r}}$, where $p_{1}<p_{2}<\cdots <p_{r}$ and $\alpha_{i} \geq 1$ for all $1\leq i\leq r$.\\
\textbf{\underline{Claim 1.}} $r\leq 3$. \\
\textbf{\underline{Proof of Claim 1.}} Suppose $r\geq 4$. Then, $\pi(G)=\{p_{1}, p_{2}, p_{3}, p_{4}\}$. Since, $G$ is solvable then it has Hall subgroups. Let $\pi_{1}=\{p_{1}\}, \pi_{2}=\{p_{2}\}, \pi_{3}=\{p_{2}, p_{4}\}, \pi_{4}=\{p_{1}, p_{3}\}$ and $H_{\pi_{j}}$ be the corresponding Hall $\pi_{j}$-subgroups of $G$, where $1 \leq j \leq 4$. Then $\Gamma_{S}(G)$ contains a $4$-vertex cycle: $(H_{\pi_{1}}, H_{\pi_{2}}, H_{\pi_{4}}, H_{\pi_{3}}, H_{\pi_{1}})$, where $o(H_{\pi_{1}})=p_{1}^{\alpha_{1}}, o(H_{\pi_{2}})=p_{2}^{\alpha_{2}}, o(H_{\pi_{3}})=p_{2}^{\alpha_{2}}p_{4}^{\alpha_{4}}, o(H_{\pi_{4}})=p_{1}^{\alpha_{1}}p_{3}^{\alpha_{3}}$.
Hence, by Claim-1, $o(G)$ can have at most $3$ distinct prime divisors.\\
\textbf{\underline{Claim 2.}} if $r=3$ then $\alpha_{1}=\alpha_{2}=\alpha_{3}=1$\\
\textbf{\underline{Proof of Claim 2.}} Let there exist at least one $\alpha_{i}>1$ (say $\alpha_{1}>1$). Let $\pi_{1}=\{p_{1}\}, \pi_{2}=\{p_{2}\}, \pi_{3}=\{p_{2}, p_{3}\}$. Then $o(H_{\pi_{1}})=p_{1}^{\alpha_{1}}, o(H_{\pi_{2}})=p_{2}, o(H_{\pi_{3}})=p_{2}p_{3}$. Clearly, $\Gamma_{S}(G)$ contains a $4$-vertex cycle:$(H_{\pi_{1}}, H_{\pi_{2}}, P_{1}, H_{\pi_{3}}, H_{\pi_{1}})$, where $o(P_{1})=p_{1}$ and $P_{1}\subset H_{\pi_{1}}$. This completes the proof of Claim-2.\\
By Claim-1 and Claim-2 we conclude that if $r=3$ then $G$ must be a group of order $p_{1}p_{2}p_{3}$, where $p_{1}<p_{2}<p_{3}$.\\
\textbf{\underline{Claim 3.}} if $r=2$, then $\alpha_{i}=1$ for some $i=1, 2$\\
\textbf{\underline{Proof of Claim 3.}}
Let $o(G)=p_{1}^{\alpha_{1}}p_{2}^{\alpha_{2}}$ with both $\alpha_{i}>1$. Then $\Gamma_{S}(G)$ contains a $4$-vertex cycle:$(H_{\pi_{1}}, H_{\pi_{2}}, P_{1}, P_{2}, H_{\pi_{1}})$, where $\pi_{i}=\{p_{i}\}$ and $o(H_{\pi_{i}})=p_{i}^{\alpha_{i}}$ for $i=1, 2$; $P_{i}\subset H_{\pi_{i}}$
with order  $o(P_{i})=p_{i}$ for $i=1, 2$. Hence by Claim-3, there exists at least one $\alpha_{i}$ (say $\alpha_{2}$) which is greater than $1$.\\
By Claim-3, $o(G)=p_{1}^{\alpha_{1}}p_{2}$, where $\alpha_{1}\geq 1$. 
Therefore, by Claim-1, Claim-2 and Claim-3 we conclude that $G$ is either a group of order $p_{1}p_{2}p_{3}$, where $p_{1}<p_{2}<p_{3}$ or $G$ has order $p_{1}^{k}p_{2}$, where $k \geq 1$.

\end{proof}

\begin{theorem}
Let $S_{n}, A_{n}$ be the symmetric group, alternating group on $n$-symbols respectively. Then $\Gamma_{S}(S_{n})$ is a chordal graph if and only if $n=3 $.\\
On the other hand, $\Gamma_{S}(A_{n})$ is chordal if and only if $n \leq 4$.
\end{theorem}

\begin{proof}
\textbf{(Part-I: Symmetric Groups)}\\
Let $n=4$. Then $\Gamma_{S}(S_{4})$ contains a $4$-vertex induced cycle $((C_{2})^{2}, S_{3}, D_{4}, C_{3}, (C_{2})^{2})$, where $(C_{2})^{2}=\langle (1, 3), (2, 4)\rangle, S_{3}=\langle (1, 2, 3), (1, 2)\rangle, D_{4}=\langle (1, 2, 3, 4), (1, 3)\rangle, C_{3}=\langle (1, 2, 3) \rangle$. Hence, $\Gamma_{S}(S_{4})$ is not chordal. So, if $n \geq 4$, then $\Gamma_{S}(S_{n})$ is not chordal. Therefore $n=3$.\\
For converse, using the similar argument as in Theorem \ref{th_sym} $\Gamma_{S}(S_{3})$ is a complete graph and hence it is chordal graph.\\
\textbf{Part-II: Alternating Groups}\\
We claim that if $n \geq 5$ then $\Gamma_{S}(A_{n})$ is not chordal.\\
Let $n=5$. $\Gamma_{S}(A_{5})$ contains a $4$-vertex cycle: $(D_{5}, A_{4}, C_{5}, (C_{2})^{2}, D_{5})$, where $D_{5}=\langle (1, 2, 3, 4, 5), (2, 5)(3, 4)\rangle, A_{4}=\langle (1, 2)(3, 4), (1, 2, 3) \rangle, C_{5}=\langle (1, 2, 3, 4, 5) \rangle, (C_{2})^{2}=\{e, (1, 2)(3, 4), (1, 3)(2, 4), (1, 4)(2, 3)\}$. Thus, $\Gamma_{S}(A_{5})$ and hence $\Gamma_{S}(A_{n})$ (where, $n >5$) is not a chordal graph.\\
This gives $n \leq 4$.\\
Conversely, If $n=3$, then $A_{3}$ is nothing but the cyclic group $C_{3}$. Since, $\Gamma_{S}(C_{3})$ is null graph so it is a chordal graph.\\
Next consider the group $A_{4}$. The non-trivial subgroups of $A_{4}$ are: $C_{2}\times C_{2}, C_{2}, C_{3}$. Then by Theorem \ref{th_sym} $\Gamma_{S}(A_{4})$ is cograph. So we need to check only the existance of any induced $4$-vertex cycle in $\Gamma_{S}(A_{4})$.\\ 
To construct a cycle in $\Gamma_{S}(A_{4})$, the three consecutive vertices must be $C_{2}, C_{3}, C_{2}\times C_{2}$. So, if we add another vertex then it must be either $C_{2}$ or $C_{3}$; which is adjacent to the previously added vertex. Hence, such induced cycle does not exist in $\Gamma_{S}(A_{4})$; this implies $\Gamma_{S}(A_{4})$ is chordal.
\end{proof}

\begin{theorem}
Let $D_{n}$ be the dihedral group with $2n$ elements. Then $\Gamma_{S}(D_{n})$ is chordal if and only if $n$ is prime power.
\end{theorem}

\begin{proof}
Let $\Gamma_{S}(D_{n})$ is a chordal graph. \\
We first suppose that $n$ has $3$ or more distinct prime divisors. Then $\Gamma_{S}(D_{n})$ contains an induced cycle $(\langle a^{p} \rangle, \langle a^{qr} \rangle, \langle a^{p}, b \rangle$, $\langle a^{r} \rangle, \langle a^{p} \rangle)$, where $p, q, r$ are $3$ distinct primes. So, $n$ can have at most two distinct prime divisors.\\
If possible let $n$ has at least $2$ prime divisor $p, q$. Then $\Gamma_{S}(D_{n})$ contains a cycle: $(\langle a^{p}\rangle, \langle a^{q}, b \rangle, \langle a^{p}, ab \rangle, \langle a^{q}\rangle, \langle a^{p}\rangle)$. Therefore $n$ is prime power.\\
Converse Part:\\
If $n$ is power of a prime then using the similar argument as done in Theorem \ref{th_dih}, we can prove that any cycle of length $4$ and above in $\Gamma_{S}(D_{n})$ contains a chord. Hence, there does not exist any chordless induced cycle of length $4$ and above in $\Gamma_{S}(D_{n})$. Therefore, $\Gamma_{S}(D_{n})$ is a chordal graph.
\end{proof}

%%%%%%%%%%%%%%%%%%%%%%%%%%%%%%%%%%%%%%%%%%%%%%%%%%%%%%%%

\section{Some additional groups whose intersection subgroup graph is a cograph or a chordal graph or both}
\subsection{Sporadic simple group}
In this section we consider $26$ sporadic groups, namely, the five Mathieu groups ($M_{11}, M_{12}, M_{22}, M_{23}, M_{24}$), three Conway groups ($Co_{1}, Co_{2}, Co_{3}$), four Janko groups ($J_{1}, J_{2}, J_{3}, J_{4}$), three Fischer groups ($Fi_{22}, Fi_{23}, Fi_{24}$), Higman-Sims group($HS$), the Held group($He$), the McLaughlin group ($M^{c}L$), the Rudavalis group ($Ru$), the Suzuki group ($Suz$), the O'Nan group ($O'N$), the Harada-Norton group ($HN$), the Lyons group ($Ly$), the Thompson group ($Th$), the Baby Monster group ($B$) and the Monster group ($M$).
\begin{theorem}
Let $G$ be a sporadic simple group. Then $\Gamma_{S}(G)$ is neither a cograph nor a chordal graph.
\end{theorem}

\begin{proof}
According to the informations available at $\mathbb{ATLAS}$, the Mathieu group $M_{11}$ contains $S_{5}$ as the maximal subgroup. Since, $\Gamma_{S}(S_{5})$ is neither a cograph nor a chordal graph so is $\Gamma_{S}(M_{11})$. The other sporadic groups contain $M_{11}$ as its subgroup except $J_{1}, J_{2}, J_{3}, M_{22}, He, Ru$ and $Th$; this implies the intersection subgroup graph of those groups are neither a cograph nor a chordal graph. \\
For the exceptional seven groups we look for the subgroup whose intersection subgroup graph is neither a cograph nor a chordal graph.  $M_{22}$ contains $A_{7}$; $J_{1}, J_{2}$ contain $A_{5}$; $J_{3}$ contains $A_{6}$; $He$ contains $S_{7}$; $Ru$ contains $A_{8}$ and $Th$ contains $S_{5}$.\\
This completes the proof of the theorem.
\end{proof}

Before going to the subsequent subsections, first we want to recall Mih\`ailescu Theorem, which is a proof of the famous Catalan's conjecture, from number theory.
\begin{theorem}
The only solution to the equation $x^{a}- y^{b}=1$ in the natural number for $a, b>1$ and $x, y>0$ is $x=3, y=2, a=2, b=3$.
\end{theorem}
We use this theorem to deal with the number theoretic obstacles in the upcomming two sections.
\subsection{Classical group of Lie type of rank $1$ and $2$}
The groups $\PSL(2, q)$, $\PSU(3, q)$, $\PSL(3, q)$ and $\PSp(4, q)$ are the classical groups of rank $1$ and $2$. In this section, we explore the groups $G$ such that $\Gamma_{S}(G)$ is a cograph. 
\begin{theorem}
\label{even_case}
Let $G\cong \PSL(2, q)$ with $q$ even. Then $\Gamma_{S}(G)$ is cograph if and only if $q=2$.
\end{theorem}

\begin{proof}
First we find the case when $\Gamma_{S}(G)$ is cograph:\\
For $q=2^{f}~(q \geq 4)$, $G=\PSL(2, q)$ contains subgroup $(C_{2})^{f}\rtimes C_{q-1}$. Clearly $f=2$; elsewhere $\Gamma_{S}(G)$ contains a path $C_{2}\times C_{2}\times \{0\}, \{0\}\times \{0\}\times C_{2}, \{0\}\times C_{2}\times \{0\}, C_{2}\times \{0\}\times C_{2}$. But if $q=4$ then $G=\PSL(2, 4)\cong A_{5}$. Therefore, $\Gamma_{S}(A_{5})$ and hence $\Gamma_{S}(G)$ is not a cograph. \\
If $q=2$, then $\PSL(2, 2)\cong S_{3}$. So, $\Gamma_{S}(\PSL(2, 2))$ is cograph.\\
Converse part is obvious.
\end{proof}

\begin{theorem}
\label{odd_case}
Let $G\cong \PSL(2, q)$ with $q=p^{f}$($p$ odd prime). Then $\Gamma_{S}(G)$ is a cograph if and only if $q=3$.
\end{theorem}

\begin{proof}
For $q=p^{f}(\geq 5)$, $\PSL(2, q)$ contains subgroup $(C_{p})^{f}$. By similar argument as done in Theorem \ref{even_case} we conclude $f=1$ or $2$.\\
If $f=1$ (or $q=p$), then $\PSL(2, q)$ contains the dihedral groups $D_{p-1}$ (if $q \geq 13$) and $D_{p+1}$( for $p \neq 7, 9$). Since $p$ is odd, then one of $p\pm 1$ is not of the form power of some prime. So, either $\Gamma_{S}(D_{p+1})$ or $\Gamma_{S}(D_{p-1})$ is not a cograph. This left the only case $p=5$. For $p=5$, $\PSL(2, 5)$ is the alternating group $A_{5}$. We know that $\Gamma_{S}(A_{5})$ is not a cograph.\\
If $f=2$ (or $q=p^{2}(\geq 5)$),  then $\PSL(2, q)$ contains the dihedral groups $D_{p^{2}-1}$ (if $q \geq 13$) and $D_{p^{2}+1}$( for $p \neq 9$). Clearly, if $p$ odd, then $p^{2}-1$ is composite which is not a prime power. So, $\Gamma_{S}(D_{p^{2}-1})$ is not cograph. On the other hand, if $q=11$ then $\PSL(2, q)$ contains $D_{q+1}$ or $D_{12}$. But $\Gamma_{S}(D_{12})$ is not cograph. Now we have to check for $q=3, 9$.\\
If $q=3$ then $\PSL(2, 3)$ is $A_{4}$, whose intersection subgroup graph is a cograph.\\
If $q=9$ then $\PSL(2, 9)$ contains $\PGL(2, 3)$ which is $S_{4}$. As $\Gamma_{S}(S_{4})$ is not cograph so is $\Gamma_{S}(\PSL(2, 9))$.
Converse is clear.
\end{proof}

From the above theorem it is clear that if $G=\PSL(2, q)$, then $\Gamma_{S}(G)$ is cograph if and only if $q \in \{2, 3\}$. Thus a question arise in case of chordality :

\begin{problem}
Find the values of $q$ for which $\PSL(2, q)$ is chordal.
\end{problem}

\begin{theorem}
Let $G\cong \PSp(4, q)$. Then $\Gamma_{S}(G)$ is never be a cograph.
\end{theorem}

\begin{proof}
Two cases arise.\\
\textbf{Case 1.} $q$ is odd.\\
$\PSp(4, q)$ contains subgroups of the form $H\cong C_{q\pm 1}\times C_{(q\pm 1)/2}$. Clearly, as $q\pm 1$ is divisible by $2$ so the numbers $(q\pm 1)/2$ must be prime power. On the other hand, the numbers $q\pm 1$ must be power of $2$; otherwise $H$ contains a subgroup $C_{2p}\times C_{p}$ or $C_{2p}\times C_{2}$ with odd prime $p$. But by Theorem \ref{th_nil_co} the intersection subgroup graph of both are not cograph. Now one of the numbers $(q\pm 1)/2$ is power of $2$ as one of them is divisible by $2$. So, one of $q\pm 1$ must be equal to $4$; elsewhere there exists a subgroup $C_{4}\times C_{2p}$(where $p$ odd). Thus, the remaining cases are: $q=3, 5$.\\
If $q=5$, $\PSp(4, 5)$ contains $S_{5}$ and for $q=3$, $\PSp(4, 3)$ contains $S_{6}$ as a subgroup. Since neither $\Gamma_{S}(S_{5})$ nor $\Gamma_{S}(S_{6})$ is a cograph so in this case we conclude there does not exist $q$ for which $\Gamma_{S}(\PSp(4, q))$ is a cograph.\\
\textbf{Case 2.} $q\geq4$ is even.\\
Then $\PSp(4, q)$ contains a subgroup $C_{q\pm 1}\times C_{q\pm 1}$. Then both the numbers $q\pm 1$ must be prime power; otherwise $\PSp(4, q)$ contains a nilpotent subgroup say $H$ such that $\Gamma_{S}(H)$ is not a cograph. Then by Catalan's conjecture we have $q \in \{2, 4, 8\}$. If $q=2$, $\PSp(4, 2)\cong S_{6}$. If $q=4$, $\PSp(4, 4)$ contains $S_{6}$ and for $q=8$, $\PSp(4, 8)$ contains $\PSL(2, 8)$ as a subgroup. Since none of them has intersection subgroup graph which is a cograph so for even values of $q$, $\Gamma_{S}(\PSp(4, q))$ is never be a cograph.
\end{proof}

\begin{theorem}
Let $G\cong \PSL(3, q)$. Then $\Gamma_{S}(G)$ is never be a cograph.
\end{theorem}

\begin{proof}
According to \cite{wilson}, $G$ contains subgroups $C_{q+1}, C_{q-1}\times C_{q-1/gcd(3, q-1)}$. Here we consider two cases based on $q$.\\
\textbf{Case 1.} $q$ is odd.\\
Since $q$ is odd and $\Gamma_{S}(C_{q-1})$ is a cograph, then $q-1$ must be a power of $2$ or $2^{r}p^{s}$. For the later case, we get a subgroup $C_{2p}\times C_{2}$ contained in $C_{q-1}\times C_{q-1/gcd(3, q-1)}$. Since $\Gamma_{S}(C_{2p}\times C_{2})$ is not a cograph (due to Theorem \ref{th_nil_co}), $q-1$ must be a power of $2$. Thus either $q=3$ or $q-1=2^{r}$ with $r>1$. \\
If $q-1=2^{r}$, then $q=2^{r}+1$ and $q+1=2(2^{r-1}+1)$. Now the pairs $(q-1, q)$ is a solution to Catalan's conjecture; which gives $q=9$ or $q$ is a prime say $p_{1}$. As $q=2^{r}+1$ is prime then $r$ must be power of $2$. So, $r-1$ is odd and then $6|(q+1)$. \\
On the other hand, if $\Gamma_{S}(C_{q+1})$ is cograph this implies $q+1$ is either power of $2$ or $2^{m}3^{n}$ (with $m, n \geq 1$). But both $q+1, q-1$ cannot be power of $2$ as in that case the pair $(q-1, q+1)$ will be a solution to Catalan's conjecture and this contradicts Mih\`ailescu's Theorem. Thus $q+1=2^{m}3^{n}$ and $q-1=2^{r}$. Now, $(q+1)-(q-1)=2$ gives $2^{m}3^{n}-2^{r}=2$ i.e. $2^{m-1}3^{n}-2^{r-1}=1$. Then the only solutions are $(2^{m-1}3^{n}, 2^{r-1})=(3, 2)$ or $(9, 8)$. Hence $q=5, 17$. Thus we have to check the cases $q \in \{3, 5, 9, 17\}$.\\
If $q=3$, $\PSL(3, 3)$ contains $D_{36}$ as a subgroup; for $q=5$, $\PSL(3, 5)$ contains $S_{5}$. Again for $q=9$, $\PSL(3, 9)$ contains $\PSL(3, 3)$ and if $q=17$, $\PSL(3, 17)$ contains either $\PSL(2, 17)$ or $\PGL(2, 17)$ (by \cite{bloom}). Clearly $\Gamma_{S}(\PSL(2, 17))$ is not a cograph by Theorem \ref{odd_case}. And $\PGL(2, 17)$ contains $D_{18}$ (by \cite{giudici}). Hence, for none of the values of odd $q$, $\Gamma_{S}(\PSL(3, q))$ is a cograph.\\
\textbf{Case 2.} $q$ is even.\\
As $q-1$ is odd so it is either power of an odd prime or three times of power of an odd prime. Moreover, $G$ also contains a subgroup $C_{q+1}\times C_{q-1/gcd(3, q-1)}$.\\ Let \textbf{$q \neq 4$}:\\ Then $q+1$ must be prime power as $gcd(q-1, q+1)=1$ for $q$ even. \\
Let $q-1=3p^{r}$ and $q+1=p_{1}^{s}$ then $G$ contains a nilpotent subgroup $C_{3p}\times C_{p_{1}}$, whose intersection subgroup graph is not a cograph due to Theorem \ref{th_nil_co}. This implies $q-1=p^{r}$. \\
Now, both $q\pm 1$ are power of some odd prime. Due to Catalan's conjecture if $q \neq 2, 8$, then $q+1$ must be prime say $p_{1}$. Since $q=2^{k}$ then $q+1=2^{k}+1$ is a prime $p_{1}$. This implies $k$ must be power of $2$. If $k \geq 8$ then $q-1=2^{k}-1$ is divisible by $3$ distinct prime divisors of $2^{8}-1$. Hence, $k=2$ or $4$ and therefore $q \in \{2, 4, 8, 16\}$.\\
If $q=2$, $\PSL(3, 2) \cong \PSL(2, 7)$. As, $\Gamma_{S}(\PSL(2, 7))$ is not cograph so is $\PSL(3, 2)$.\\
If $q=4$, $\PSL(3, 4)$ contains $A_{6}$, whose intersection subgroup graph is not cograph.\\
For $q=8$, $\PSL(3, 8)$ contains $\PSL(3, 2)$. On the other hand, for $q=16$, $\PSL(3, 16)$ contains $A_{6}, A_{7}$, whose intersection subgroup graph is not cograph. So, there does not exist any even $q$ such that $\Gamma_{S}(\PSL(3, q))$ is a cograph. 
\end{proof}

\begin{theorem}
\label{psu_ev_co}
Let $G\cong \PSU(3, q)$, where $q(\geq 4)$ is even. Then $\Gamma_{S}(G)$ is never be a cograph.
\end{theorem}

\begin{proof}
If $q$ is even, then $gcd(q+1, q-1)=1$. Then $G$ contains a subgroup $C_{q\pm 1}\times C_{(q+1)/gcd(3, q+1)}$. If $gcd(3, q+1)=3$ then $\PSU(3, q)$ contains a subgroup $C_{q+1/3}\times \PSL(2, q)$. As, $\PSL(2, q)$ is cograph for $q=2$ so if $q\geq 4$, $\Gamma_{S}(\PSU(3, q)$ is not a cograph. \\
Now if $gcd(3, q+1)=1$ then $q+1$ must be prime due to Catalan's conjecture. Using same argument we get $q-1$ must be prime or $1$, this leads to contradiction unless $q=4, 8$. But $\PSU(3, 4)$ contains $A_{5}$. On the other hand, $\PSU(3, 8)$ contains $A_{6}, A_{7}$. As neither $\Gamma_{S}(A_{6})$ nor $\Gamma_{S}(A_{7})$ is a cograph, so $\Gamma_{S}(PSU(3, 8))$ is not a cograph. \\
\end{proof}

\begin{remark}
Let $G\cong \PSU(3, q)$, where $q$ is a power of odd prime.\\
Then $\PSU(3, q)$ contains cyclic subgroups of order $(q-1)(q+1)/gcd(3, q+1)$. Clearly both $q-1, (q+1)/gcd(3, q+1)$ are even numbers. Thus, $(q-1)(q+1)/gcd(3, q+1)=2^{a}p^{b}$, where $p$ is an odd prime and $a>1, b \geq 0$. \\
If $b=0$, we get $q=3$ or $5$. But if $q=3$, $\PSU(3, 3)$ contains a subgroup $C_{4}. S_{4}$ (third non-split extention by $C_4$ of $S_4$ acting via $(S_4/A_4)=C_2$) which contains $D_{6}$ as a subgroup whose intersection subgroup graph is not a cograph. If $q=5$, $\PSU(3, 5)$ contains $A_{7}$.\\
If $b\geq 1$, one of $q-1, (q+1)/gcd(3, q+1)$ is power of $2$ only. Without loss of generality, let $q-1=2^{k}$. Then, $q=2^{k}+1$, which is either an odd prime or $9$ (by Catalan's conjecture). Then, $q+1=2(2^{k-1}+1)$. If $q\neq 3, 9$ then $2^{k-1}+1$ is either $p^{b}$ or $3p^{b}$. \\
If \textbf{$2^{k-1}+1=3p^{b}$}, then $\PSU(3, q)$ contains a subgroup (say $H_{1}$) isomorphic to $C_{q+1}\times C_{(q+1)/gcd(3, q+1)}$. Since $6p|(q+1)$, so $C_{6p}\times C_{2p}$ is contained in $H_{1}$. As $\Gamma_{S}(C_{6p}\times C_{2p})$ is not cograph so is $\Gamma_{S}(H_{1})$. Hence, $2^{k-1}+1=p^{b}$. If $b>1$, by using Catalan's conjecture we get $q=17$. \\
If $b=1$, we have $2^{k}+1=p_{1}(say), 2^{k-1}+1=p$. This gives a linear diophantine equation $2p \setminus p_{1}=1$. This equation has infinitely many solutions including $(p, p_{1})=(3, 5), (7, 13), (19, 37), (37, 73)$ etc. which gives $q=5, 13, 37,...etc$.
\end{remark}

\begin{problem}
Does there exist any odd prime $p$ such that $\Gamma_{S}(\PSU(3, q))$ (for $q=p^{n}$) is a cograph?
\end{problem}

\subsection{Exceptional groups of Lie type and Tits group}
The groups ${}^{2}{B_{2}}(q)$ where $q=2^{2e+1}~(q \geq 8)$, ${}^{2}{G_{2}}(q)$, where $q=3^{2e+1}~(q \geq 9)$, $G_{2}(q)$, ${ }^{3}{D_{4}}(q)$, $F_{4}(q)$, ${ }^{2}{F_{4}}(q)$, $E_{6}(q)$, ${ }^{2}{E_{6}}(q)$, $E_{7}(q)$ and $E_{8}(q)$ are the exceptional groups of Lie type and the Tits group is ${}^{2}{F_{4}}(2)'$.\\
The main theorem of this section is following:
\begin{theorem}
Let $G$ be an exceptional Lie type simple group. Then $\Gamma_{S}(G)$ is cograph if and only if $G={}^{2}{B_{2}}(q)$ where $q=2^{2e+1}~(q \geq 8)$ such that the numbers $q-1, q\pm \sqrt{2q}+1$ is either prime power or of the form $p^{a}q^{b}~(a, b\geq 1)$.
\end{theorem}

We prove this theorem with the help of  following three lemmas.

\begin{lemma}
\label{Ree}
Let $G$ be the Ree group ${}^{2}{G_{2}}(q)$, where $q=3^{2e+1}$. Then $\Gamma_{S}(G)$ is not a cograph.
\end{lemma}

\begin{proof}
The centralizer of involution in Ree group ${}^{2}{G_{2}}(q)$ forms the subgroup $C_{2}\times \PSL(2, q)$ (see \cite{ward}), which contains subgroups $C_{2}\times C_{(q\pm 1)/2}$. This concludes that the numbers $(q\pm 1)/2$ are either power of $2$ or power of some odd prime.\\
If both are power of $2$, then we get a solution to Catalan's conjecture. This contradicts the result of Mih\`ailescu's Theorem.\\
Now if both are power of odd primes using similar argument we conclude there does not exist such $q$. Also, they are of opposite parity so both of them cannot be odd prime.\\
On the other hand if one is power of odd prime and another is power of $2$, then we must have $\dfrac{q+1}{2}=9$ and $\dfrac{q-1}{2}=8$ (by  Mih\`ailescu's Theorem) and $q=17$. But this is impossible as $q$ is odd power of $3$. Similarly if one is an odd prime and other is power of $2$ there does not exist any $q$ which is odd power of $q$.\\
Thus, there does not exist any $q=3^{2e+1}$ such that $\Gamma_{S}(G)$ is a cograph.
\end{proof}

\begin{problem}
Does there exist $q$ such that $\Gamma_{S}({}^{2}{G_{2}}(q))$ is chordal?
\end{problem}

\begin{lemma}
\label{Sz}
Let $G\cong {}^{2}{B_{2}}(q)$, where $q$ is odd prime power of $2$ ($q \geq 8$). Then $\Gamma_{S}(G)$ is cograph if and only if $q \neq 8$ such that the numbers $q-1, q\pm \sqrt{2q}+1$ are either prime power or of the form $p^{a}q^{b}$, where $a, b \geq 1$. \\
Moreover, there does not exist any $q$ such that $\Gamma_{S}(G)$ is chordal.
\end{lemma}

\begin{proof}
 First we suppose that $q>8$. Then the maximal cyclic subgroups are of order $4, q-1, q\pm \sqrt{2q}+1$. Note that these $4$ numbers are pairwise coprime and the last $3$ numbers are odd. These $4$ subgroups intersect trivially. Thus, any subgroup of $G$ is contained in only one maximal cyclic subgroup. This implies if $\Gamma_{S}(G)$ is cograph then each maximal subgroups are also so. Hence the numbers $q-1, q\pm \sqrt{2q}+1$ are either prime power or of the prescribed form in the statement of the theorem.\\
Converse Part is obvious.\\
If $q=8$, then the Sylow $2$-subgroup of $G$ is $(C_{2})^{3}._{84}(C_{2})^{3}$ (the Small group $(64, 82)$), which contains a subgroup $C_{2}^{2}\times C_{4}$. Then $\Gamma_{S}(C_{2}^{2}\times C_{4})$ contains a path $C_{2}\times C_{2}\times \{0\}, \{0\}\times \{0\}\times C_{4}, \{0\}\times C_{2}\times \{0\}, C_{2}\times \{0\}\times C_{2}$. Hence in this case $\Gamma_{S}(G)$ is not a cograph.\\
For the chordality, if $q=8$ then $G$ contains a subgroup $C_{2}^{2}\times C_{4}$. And $\Gamma_{S}(C_{2}^{2}\times C_{4})$ contains a $4$-cycle: $C_{2}\times C_{2}\times \{0\}, \{0\}\times \{0\}\times C_{4}, \{0\}\times C_{2}\times \{0\}, \{0\}\times \{0\}\times C_{2}, C_{2}\times C_{2}\times \{0\}$. Thus, $\Gamma_{S}(G)$ is not chordal.\\
Now assume $q>8$. Then at least one of $q-1, q \pm \sqrt{2q}+1$ is composite. Let $q+\sqrt{2q}+1$ is composite. Then consider $4$ subgroups $H_{1}, H_{2}, H_{3}, H_{4}$ such that $H_{1}\cong C_{4}, H_{3}\cong C_{2}, H_{2}\cong C_{q+\sqrt{2q}+1}, H_{4}\cong C_{p}$, where $p$ divides $q+\sqrt{2q}+1$. It is clear that these $4$ subgroups form the $4$ vertices of a cycle in $\Gamma_{S}(G)$. Similarly, if either $q-1$ or $q-\sqrt{2q}+1$ is composite we get a $4$-cycle in $\Gamma_{S}(G)$. Therefore, $\Gamma_{S}(G)$ is never chordal.

\end{proof}

\begin{lemma}
\label{ex_gr}
Let $G$ be the exceptional groups of Lie type (except the Suzuki group ${}^{2}{B_{2}}(q)$, where $q=2^{2n+1}~(\geq 8)$ and the Ree group ${}^{2}{G_{2}}(q)$, where $q=3^{2e+1}$) and the Tits group ${}^{2}{F_{4}}(2)'$. Then $\Gamma_{S}(G)$ is neither a cograph nor a chordal graph.
\end{lemma}

\begin{proof}
A) Let $G$ be the Tits group ${}^{2}{F_{4}}(2)'$. Since $G$ contains $\PSL(2, 25)$ whose intersection subgroup graph is not a cograph. So, $\Gamma_{S}(G)$ is neither a cograph nor a chordal graph.\\
B) Let $G$ be the group ${}^{2}{F_{4}}(q)$. This group encloses the Tits group ${}^{2}{F_{4}}(2)'$. Hence in this case $\Gamma_{S}(G)$ cannot be either a cograph or a chordal graph.\\
C)Let $G\cong F_{4}(q)$. Then for any values of $q$, $G$ contains one of the following subgroups: alternating groups $A_{7}, A_{9}, A_{10}$, the Mathieu group $M_{11}$ or the classical groups $\PSL(2, 7), \PSL(2, 25)$. But in any of the cases the intersection subgroup graph is neither a cograph nor a chordal graph.\\
D) Let $G$ be one of the group $E_{6}(q), {}^{2}{E_{6}}(q), E_{7}(q)$ and $E_{8}(q)$. Then in each of the cases $F_{4}(q)$ contained in $G$ as a subgroup. From C) we conclude that $\Gamma_{S}(G)$ is neither a cograph nor a chordal graph.\\
E) Now consider the group $G_{2}(q)$. Then for different values of $q$, $G$ contains one of the following groups as a subgroup: alternating groups $A_{5}, A_{6}, A_{7}$; the Janko groups $J_{1}, J_{2}$; the classical groups $\PSL(2, 7), \PSL(2, 8), \PSL(2, 13)$. But as none of these groups can have the intersection subgroup graph which is either cograph or chordal so $\Gamma_{S}(G)$ can never be a cograph as well as chordal graph.\\
F) If $G\cong {}^{3}{D_{4}}(q)$, then $G_{2}(q)$ is contained in $G$ as a subgroup. This implies $\Gamma_{S}(G)$ is neither a cograph nor a chordal graph.
\end{proof}

\subsection{Torsion-free nilpotent group}
In this section the major conclusion is that the intersection subgroup graph of a torsion-free nilpotent group is neither a cograph nor a chordal graph. 
\begin{theorem}
\label{th_tor_abl}
Let $G$ be a torsion-free abelian group. Then $\Gamma_{S}(G)$ is neither a cograph nor a chordal graph.
\end{theorem}

\begin{proof}
Let, $G$ be a torsion-free abelian group and $H$ be a finitely generated subgroup of $G$. Then $H$ is free abelian group which is direct product of countable copies of $\mathbb{Z}$. Now, we assert that $\Gamma_{S}(H)$ is neither a cograph nor a chordal graph. To prove this assertion first we prove the following lemma.
\begin{lemma}
\label{lm_tor}
 $\Gamma_{S}(\mathbb{Z})$ is neither a cograph nor a chordal graph.
\end{lemma}

\begin{proof}
 Clearly, every non-trivial subgroup of $\mathbb{Z}$ is cyclic group of the form $(m\mathbb{Z}, +)$. Clearly, $m\mathbb{Z}\cap n\mathbb{Z}=\{0\}$ if and only if $(m, n)=1$ for $m, n\in \mathbb{Z}$.
Now, $\Gamma_{S}(\mathbb{Z})$ contains a $4$-vertex path $(10\mathbb{Z}, 3\mathbb{Z}, 5\mathbb{Z}, 6\mathbb{Z})$. Hence, $\Gamma_{S}(\mathbb{Z})$ is not a cograph.\\
Analogously, $\Gamma_{S}(\mathbb{Z})$ contains a $4$-vertex induced cycle: $(2\mathbb{Z}, 3\mathbb{Z}, 4\mathbb{Z}, 9\mathbb{Z}, 2\mathbb{Z})$; this proves that $\Gamma_{S}(\mathbb{Z})$ is not a chordal graph.
\end{proof}
If $H$ is a finitely generated abelian group then it is direct product of finite copies of $\mathbb{Z}$. Since, $\Gamma_{S}(\mathbb{Z})$ is neither a cograph nor a chordal graph (according to the Lemma \ref{lm_tor}) so is $\Gamma_{S}(H)$. Again, as $H$ is a subgroup of $G$ so $\Gamma_{S}(G)$ is neither a cograph nor a chordal graph.
\end{proof}

\begin{theorem}
Let $G$ be a torsion-free nilpotent group. Then $\Gamma_{S}(G)$ is neither a cograph nor a chordal graph.
\end{theorem}

\begin{proof}
Let $G$ be a torsion-free nilpotent group. Then the center of $G$ (or, $Z(G)$) is a torsion-free abelian group. By the Theorem \ref{th_tor_abl}, the intersection subgroup graph of a torsion-free abelian group is neither cograph nor chordal. As $Z(G)$ is a subgroup of $G$, so $\Gamma_{S}(G)$ is neither a cograph nor a chordal graph.
\end{proof}

\begin{remark}
Let us consider the torsion-free group $\mathbb{Q}^{n}$ (for $n \geq2$). Then it contains $\mathbb{Z}^{n}$ as a subgroup. By Lemma \ref{lm_tor}, $\Gamma_{S}(\mathbb{Z})$ is neither a cograph nor a chordal graph. So, $
\Gamma_{S}(\mathbb{Q}^{n})$ is neither a cograph nor a chordal graph.
\end{remark}
\section{Acknowledgement}
The author Pallabi Manna is supported by Department of Atomic Energy (DAE), India and Santanu Mandal thanks to the University Grants Commission (UGC), India for the monetary support.

\section{Statements and Declarations} \textbf{Competing Interests:} The authors made no mention of any potential conflicts of interest.

\end{document}